\documentclass[12pt]{amsart}

\usepackage{amsmath,amssymb,amstext,amscd,amsthm,amsfonts}
\usepackage{mathtools}
\usepackage{dsfont}

\usepackage{kuvio}
\usepackage{pstricks}
\usepackage{fancyhdr}
\usepackage[dvips]{graphicx}

\usepackage{a4wide}
\usepackage[bookmarks]{hyperref}
\usepackage[numbers,sort]{natbib}
\usepackage{stmaryrd}
\usepackage{epsf}
\usepackage[absolute,overlay]{textpos}
\usepackage{everyshi,calc,eso-pic,ifthen}
\usepackage{wallpaper}
\usepackage{dcolumn}
\usepackage{pdflscape}

\newtheorem{thm}{Theorem}[section]
\newtheorem{lem}[thm]{Lemma}

\newtheorem{cor}[thm]{Corollary}
\newtheorem{prop}[thm]{Proposition}

\theoremstyle{definition}
\newtheorem{dfn}[thm]{Definition}

\theoremstyle{remark}

\newtheorem{rem}[thm]{Remark}

\newcommand{\be}{\begin{eqnarray*}}
\newcommand{\ee}{\end{eqnarray*}}
\newcommand{\ba}{\begin{align*}}
\newcommand{\ea}{\end{align*}}

\newcommand{\up}{\uparrow}

\newcommand{\N}{\mathds{N}}
\newcommand{\Z}{\mathds{Z}}

\newcommand{\R}{\mathds{R}}
\newcommand{\C}{\mathds{C}}
\newcommand{\RP}{\mathds{R}\text{P}}
\newcommand{\CP}{\mathds{C}\text{P}}
\newcommand{\HP}{\mathds{H}\text{P}}

\newcommand{\x}{\times}
\newcommand{\xto}{\xrightarrow}
\newcommand{\hto}{\hookrightarrow}
\newcommand{\ot}{\otimes}

\newcommand{\im}{\text{im}\,}
\newcommand{\id}{\text{id}}
\newcommand{\sign}{\text{sign}\,}
\newcommand{\Th}{\text{Th}}
\newcommand{\ad}{\text{ad}\,}

\newcommand{\modu}{\text{\ mod\ }}
\newcommand{\Emb}{\text{Emb}}
\newcommand{\Diff}{\text{Diff}}

\newcommand{\ga}{\alpha}
\newcommand{\gb}{\beta}
\newcommand{\gc}{\gamma}
\newcommand{\gd}{\delta}
\newcommand{\gw}{\omega}

\numberwithin{equation}{section}

\begin{document}

\title{Bordism and Projective Space Bundles}

\author{Sven F\"uhring}
\email{sven.fuehring@gmail.com}
\subjclass[2010]{Primary 53C20; Secondary 55N22}
\keywords{Bordism, Manifolds of positive scalar curvature} 
\date{May 25, 2020}

\begin{abstract}
We prove that total spaces of $\CP^2$-bundles generate the oriented cobordism ring $\Omega_*$. Combined with the \emph{surgery lemma} (\cite{gl}, \cite{sy}) this  yields a somewhat different proof of Gromov's and Lawson's theorem \cite{gl} that all simply connected non-spin manifolds of dimension $\geq 5$ carry a metric of positive scalar curvature (pscm).
\end{abstract}

\maketitle

\tableofcontents

\section{Introduction} The purpose of the present paper is to prove
\begin{thm}\label{mt}
The oriented cobordism ring $\Omega_*$ is in positive degrees generated by the total spaces of $\CP^2$-bundles with structure group $PU(3)\rtimes\Z_2$. In other words, any closed oriented manifold is cobordant to such a total space.
\end{thm}
The pivotal work \cite{thom} makes the geometric question 'which manifolds are cobordant?' accessible by translating it into an algebraic topological problem. In \cite{thom} and \cite{milodd} it is proved that $\Omega_*$ is a finitely generated abelian group without odd torsion. It turns out that the proof of Theorem \ref{mt} for the torsion of $\Omega_*$ is more complicated than the proof for $\Omega_*/\text{Tor}$. The first complete set of generators of $\Omega_*$ is given by \cite{wall} using results of \cite{roh} and \cite{dold}. Our proof for the torsion follows \cite{stolz}. Stolz proved therein that localized at 2 the kernel of the Atiyah-Bott-Shapiro orientation $\ga\colon\Omega^{Spin}_*\to ko_*$ (cf. \cite{abss}) is generated by total spaces of $\HP^2$-bundles. The crucial steps in Stolz's proof are the study of the general bundle transfer and a subtle use of the Adams spectral sequence. We carry his strategy to the oriented case. For the proof for $\Omega_*/\text{Tor}$ we use Milnor's criterion \cite{milodd} and ideas of \cite{hpeh}.

Stolz's motivation for his work \cite{stolz} arises from Riemannian geometry: His result that the kernel of $\ga$ is generated by $\HP^2$-bundles which carry a pscm is one central step in his proof that all simply connected spin manifolds with vanishing $\ga$-genus of dimension $\geq 5$ carry a pscm.
The corresponding statement in the oriented case is
\begin{thm}\label{tsc}
All simply connected non-spin manifolds of dimension $\geq 5$ carry a pscm.
\end{thm}
This was first proved in \cite{gl}. Similarly to the spin case, an important step is the insight that $\Omega_*$ is in positive degrees generated by manifolds with pscm. (In analogy to the the spin case this statement can be rephrased as follows: The kernel of the orientation map $U\colon\Omega_*\to H\Z_*$ is generated by pscm manifolds.) Gromov and Lawson showed this by using the set of generators from \cite{wall}. Since also the $\CP^2$-bundles we construct carry a pscm, Theorem \ref{mt} implies that they serve just as well for the proof of Theorem \ref{tsc}.

The basic idea for the proof of Theorem \ref{mt} is as follows. Let $G:=SU(3)\rtimes\Z_2$ and $H:=S(U(2)U(1))\rtimes\Z_2$. Then $G$ acts via matrix multiplication resp. complex conjugation transitively on $\CP^2$ with isotropy group $H$. Hence there is a fiber bundle
\begin{equation}\label{fb}
\CP^2= G/H\hto BH\xto{\pi} BG.
\end{equation}
We consider the associated transfer in bordism-homology
\[
\pi^{!}\colon\Omega_{*-4}(BG) \to \Omega_*(BH).
\]
Geometrically, $\pi^{!}$ can be identified with assigning to a manifold $f\colon N\to BG$ its pullback $f^*BH\to BH$. Combining with the collapse map $BH\to pt$ we get a map
\begin{equation}\label{psi}
\psi\colon\Omega_{n-4}(BG)\to \Omega_n.
\end{equation}
Our aim is now to prove
\begin{thm}\label{ha} $\psi$ is surjective for $n\geq 1$. \end{thm}
The statement of Theorem \ref{mt} follows immediately. As $G$ acts via isometries on $\CP^2$ the O`Neill formulas (cf. \cite{besse}, Proposition 9.70d) and a theorem of Vilms (cf. \cite{besse}, Theorem 9.59) guarantee pscm on total spaces. Hence we see that $\Omega_*$ is in positive degrees generated by manifolds with pscm.

\begin{rem}
We note that $SU(3)\rtimes\Z_2$ does not act effectively on $\CP^2$. Namely, the isometry group of $\CP^2$ is $PU(3)\rtimes\Z_2$.
There is, however, a natural map
\[
SU(3)\rtimes\Z_2\hto U(3)\rtimes\Z_2\twoheadrightarrow PU(3)\rtimes\Z_2.
\] 
Now it follows from the natuarality of transfer maps that Theorem \ref{ha} also holds for the corresponding map $\Omega_{n-4}(B(PU(3)\rtimes\Z_2))\to \Omega_n$. Hence we may assume the strucure group to be $PU(3)\rtimes\Z_2$.
\end{rem}

In the next section we give an outline of the proof of Theorem \ref{ha} for the torsion of $\Omega_*$. Then we present a recent description of the general bundle transfer (Sec. \ref{ctransfer}) including statements used in \cite{stolz}. This presentation is self-contained and avoids references to the unpublished work \cite{bor}. Next we study cohomological properties of the fiber bundle \ref{fb} (Sec. \ref{fiberbundle}). Then we describe the homology of MSO-module spectra (Sec. \ref{smso}) and use this results in Sec. \ref{splitsurjection} resp. \ref{adams} for cohomological computations resp. the study of the Adams spectral sequence. In Sec. \ref{free} we give the proof for $\Omega_*/\text{Tor}$.

In the last section we use the developed techniques to present the corresponding result for unoriented cobordism: Each manifold is cobordant to the total space of an $\RP^2$-bundle. This was proven in \cite{gschnitzer}. 
\bigskip

\noindent\textbf{Acknowledgements.}
The present paper is the English translation of my diploma thesis \emph{$\CP^2$-B\"undel und Bordismus} written at the Technical University of Munich in 2008 (cf. \cite{diplom}). I am indepted to my thesis advisor Prof. Hanke for his continuous encouragement and valuable support.

\section{Outline of the proof}

For a space $A$ let $A_+$ denote the disjoint union of $A$ with some basepoint and $\Sigma^kA$ the $k$-fold suspension. The classical Thom-Pontrjagin construction (cf. \cite{stong}, Ch. 2) establishes an isomorphism
\[
\Omega_n(A)\cong \pi_n(MSO \wedge A_+),
\]
where $MSO$ denotes the oriented Thom spectrum\footnote{In this paper we use 'spectra' as it is defined (for example) in \cite{switzer}, Ch. 8.}. Now the proof of Theorem \ref{ha} relies on the study of a map
\[
T\colon MSO\wedge\Sigma^4BG_+\to MSO
\]
that induces $\psi$ (cf. \ref{psi}) on homotopy groups, i.e.
\[
\psi=T_*\colon \pi_n(MSO\wedge\Sigma^4BG_+)\to\pi_n(MSO).
\]
To define $T$ we need the general description of transfer maps due to \cite{bg} and \cite{bor}. Since we are dealing with a non compact basespace, $BG$, we use in Sec. \ref{ctransfer} a slightly different definition of Boardman's 'umkehr map' than the original one.

The strategy is to analyze $T$ in \emph{homology} in order to understand $T$ in \emph{homotopy} by means of the Adams spectral sequence.

To proceed in this way we restate Theorem \ref{ha} as follows. Let $H\Z$ denote the integer Eilenberg-MacLane spectrum. The $n-th$ space of $MSO$ is the Thom space $\Th(\tilde{\gc_n})$ where $\tilde{\gc_n}\to BSO(n)$ denotes the universal oriented $n$-plane bundle. The Thom classes in $H^n(\Th(\gc_n);\Z)$, $n\in\N$, induce the universal Thom class $U\colon MSO \to H\Z$. The map $U$ is an isomorphism on $\pi_0$ and the zero map on higher homotopy groups. Therefore the statement $\im T_*=\ker U_*$ implies Theorem \ref{ha}.

Now we consider the fibration
\begin{equation}\label{hf}
\widehat{MSO}\xto{i}MSO\xto{U}H\Z,
\end{equation}
where $\widehat{MSO}$ denotes the homotopy fiber of $U$. The composition
\[
MSO\wedge\Sigma^4BG_+ \xto{T} MSO \xto{U} H\Z
\]
is null-homotopic because it represents an element in $H^{-4}(MSO\wedge BG_+;\Z)= 0$. Hence $T$ factors over $\widehat{MSO}$:
\[
\xgrid=16mm
\ygrid=7mm
\cellpush=6pt
\Diagram
                       &                     &  \widehat{MSO}&         &     \\
                       & \ruTo^{\widehat{T}} &  \dTo_i       &         &     \\
MSO \wedge\Sigma^4BG_+ & \rTo^T              &  MSO          & \rTo^U  & H\Z.\\
\endDiagram
\]
By means of the long exact homotopy sequence associated to \ref{hf}, $\im T_*=\ker U_*$ follows from
\begin{thm}\label{as} The map
\[
\widehat{T}_*\colon\pi_*(MSO \wedge\Sigma^4BG_+)\to\pi_*(\widehat{MSO})
\]
is surjective. \end{thm}
The surjectivity of $\widehat{T}_*$ onto $\Omega_*/\text{Tor}$ will be proved in Sec. \ref{free}.

The statement that the torsion of $\Omega_*$ is contained in $\im\widehat{T}_*$ will be shown by using the Adams spectral sequence. With the spectral sequence one can study homotopy groups of spectra and maps between them: For a spectrum $X$ and a prime $p$ the mod $p$ Adams spectral sequence converges to $\pi_*(X)/T_{\nmid p}$ (cf. \cite{switzer}, Proposition 19.12),
\[
E^{s,t}_2=Ext_{A_*}^{s,t}(\Z_p,H_*(X;\Z_p))\Longrightarrow\pi_{t-s}(X)/T_{\nmid p},
\]
where $A_*$ denotes the dual Steenrod-algebra and $T_{\nmid p}$ the torsion prime to $p$. We shall consider the mod 2 Adams spectral sequence. Unless otherwise specified homology groups are understood with $\Z_2$-coefficients.

The main bulg for the proof of Theorem \ref{as} is
\begin{thm}\label{wide} $\widehat{T}_*\colon H_*(MSO\wedge\Sigma^4BG_+) \to H_*(\widehat{MSO})$ is a split surjection of $A_*$-comodules. \end{thm}
This will be proved in Sec. \ref{splitsurjection}. Now from the functoriality of $Ext$ it follows that $T$ induces a surjection on the $E_2$-terms
\[
Ext_{A_*}^{s,t}(\Z_2,H_*(MSO\wedge\Sigma^4BG_+))\to Ext_{A_*}^{s,t}(\Z_2,H_*(\widehat{MSO})).
\]
To conclude the surjectivity on the $E_{\infty}$-terms we prove in Sec. \ref{adams}:
\begin{thm}\label{adam} The Adams spectral sequence
\[
Ext_{A_*}^{s,t}(\Z_2,H_*(MSO\wedge\Sigma^4BG_+))\Longrightarrow\pi_{t-s}(MSO\wedge\Sigma^4BG_+)/T_{\nmid 2}
\]
collapses on the $E_2$-term. \end{thm}
From this result we deduce that $\im\psi$ contains the torsion of $\Omega_*$.

In the last section we prove that also $proj\circ\psi\colon\Omega_{n-4}(BG)\to\Omega_*/\text{Tor}$ is surjective. For this statement it is sufficient to take $SU(3)$ instead of $SU(3)\rtimes\Z_2$ as structure group.

\section{The transfer map}\label{ctransfer}

We shall define the general bundle transfer and interpret it in two cases. Transfer maps with respect to arbitrary (co-)homology theories are considered in \cite{bg}. Their restrictions to compact basespaces and compact Lie groups as structure groups are not necessary. We introduce the transfer map following \cite{ebert}, Sec. 6.1 resp. \cite{gmt}, p. 8.

Let $F$ be a smooth closed manifold of dimension $k$, $B$ a (possibly infinite) connected CW-complex and $X$ a ring spectrum. Moreover let
\[
F\hto E\xto{p}B
\]
be a fiber bundle such that the tangent bundle along the fiber $\tau$ is oriented with respect to $X$. We shall construct the \emph{bundle transfer} in $X$-(co)homology associated to $E\xto{p}B$:
\[
p^!\colon X_{n-k}(B_+)\to X_n(E_+)
\]
resp.
\[
p_!\colon X^n(E_+)\to X^{n-k}(B_+).
\]

The transfer consists of the Thom isomorphism with respect to $X$ and the \emph{umkehr map} (cf. \cite{bor}). The latter is a parameterized version of the Thom-Pontrjagin collapse map: One embeds the fibers over each point of the base in some Euclidean space and takes fiberwise the collapse maps.

Let $\Emb(F,\R^{m+k})$ be the space of all smooth maps $F\hto\R^{m+k}$ equipped with the Whitney $C^{\infty}$-topology. An embedding $\iota\colon F\hto\R^{m+k}$ induces a map from the normal bundle to $\R^{m+k}$
\[
\nu^{\iota}:=(TF)^{\bot}\to\R^{m+k},\ (p,v)\mapsto \iota(p)+\iota_*(v),
\]
where $TF\subset\iota^*(T\R^{m+k})$ and $\iota_*(v)\in T_{\iota(p)}\R^{m+k}\cong_{\text{can}}\R^{m+k}$. We call $\iota$ a \emph{fat} embedding if this map restricted to the unit disk bundle is an embedding. According to Whitney's embedding theorem $\Emb(F,\R^{m+k})$ is
$(m+k-2k-2)$-connected, and from the tubular neighborhood theorem it follows that the inclusion $\Emb(F,\R^{m+k})^{\text{fat}}\hto\Emb(F,\R^{m+k})$ is a homotopy equivalence. Hence $\Emb(F,\R^{m+k})^{\text{fat}}$ is also $(m-k-2)$-connected.

Let $\pi\colon\widetilde{E}\to B$ be the $\Diff(F)$-principle bundle associated to $E\to B$. Now $G:=\Diff(F)$ acts on $\Emb(F,\R^{m+k})^{\text{fat}}$ and we consider the $\Emb(F,\R^{m+k})^{\text{fat}}$-bundle
\[
\widetilde{E}\x_G \Emb(F,\R^{m+k})^{\text{fat}}\to B.
\]
Since $\Emb(F,\R^{m+k})^{\text{fat}}$ is $(m-k-2)$-connected, there exists a section from the $m-k-2$-skeleton $B^{(m-k-2)}$ of $B$ to $\widetilde{E}_m\x_G \Emb(F,\R^{m+k})$, $\widetilde{E}_m:=\pi^{-1}(B^{(m-k-2)})$. We can chose this section such that it extends a given section on a smaller skeleton. For $m\geq k+2$ we put $B_m:=B^{(m-k-2)}$ resp. $E_m:=p^{-1}(B_m)$ and get exhaustions $B_{k+2}\subset B_{k+3}\subset B_{k+4}\subset\dots$ resp. $E_{k+2}\subset E_{k+3}\subset E_{k+4}\subset\dots$ of $B$ resp. $E$.

A section $s\colon B_m\to\widetilde{E}_m\x_G \Emb(F,\R^{m+k})$ corresponds to embeddings of the fibers over $B_m$ in $\R^{m+k}$ that depend continuously on the basepoint, in other words to a map of fiber bundles
\begin{equation}\label{emb}
j_m\colon E_m\to B_m\x\R^{m+k}
\end{equation}
over $B_m$ such that $j_m|_{p^{-1}(b)}\colon p^{-1}(b)\to\R^{m+k}$ is an embedding for all $b\in B_m$. In addition we have
\begin{equation}\label{vbed}
j_{m+1}|_{E_m}=j_m.
\end{equation}
We consider the tangent bundle along the fiber $\tau$ restricted to $E_m$ as a subbundle in $E_m\x\R^{m+k}$ and put
\[
-\tau_m:=\tau_m^{\bot}\subset E_m\x\R^{m+k}.
\]
Condition \ref{vbed} gives isomorphisms $-\tau_{m+1}|_{E_m}\cong-\tau_m\oplus\R$. Let us introduce a handy
\begin{dfn}
Given an exhaustion $A_0\subset A_1\subset A_2\subset A_3\ldots$ of some topological space $A$. A \emph{stable vector bundle} over $A$ is a sequence of vector bundles $\eta_n\to A_n$ of dimension $n$ with given isomorphisms $\eta_{n+1}|_{A_n} \cong\eta_n\oplus\R$.
\end{dfn}
We saw above that $-\tau_m\to E_m$ became a stable vector bundle.

Since $j_m|_{p^{-1}(b)}$ is a fat embedding for all $b$, we can think of $-\tau_m$ as a parameterized tubular neighborhood of $E_m$ in $B_m\x\R^{m+k}$. In this way we get a parameterized Thom-Pontrjagin collapse map
\[
t_m\colon\Sigma^{m+k}(B_{m})_+\to\Th(-\tau_m).
\]
We define a spectrum $\widehat{B}$ by $\widehat{B}_m:=\Sigma^m(B_m)_+$ with the obvious structure maps. Then $\widehat{B}$ is a cofinal subspectrum of the suspension spectrum $\Sigma^{\infty}(B_+)$, i.e. $\widehat{B}$ is homotopy equivalent to $\Sigma^{\infty}(B_+)$ via the inclusion.

Now $(t_m)$ induce a map of spectra, the umkehr map
\begin{equation}\label{kleintt}
t\colon\Sigma^k\widehat{B}\to M(-\tau).
\end{equation}

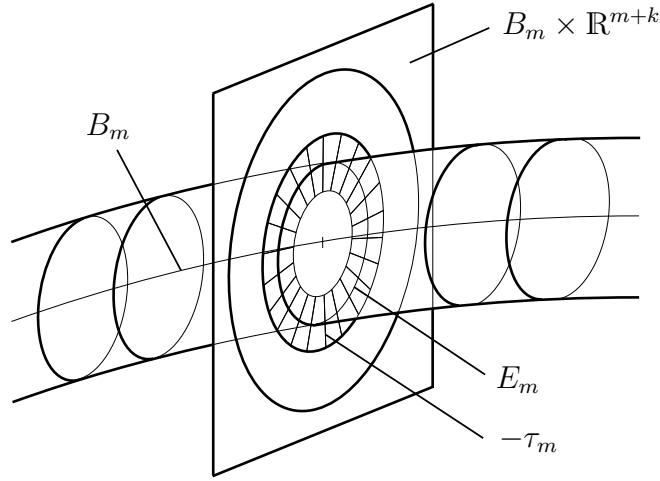
\begin{figure}

\psset{xunit=.5pt,yunit=.5pt,runit=.5pt}
\begin{pspicture}(450,744)(300,300)
{
\newrgbcolor{curcolor}{0 0 0}
\pscustom[linewidth=0.27009373,linecolor=curcolor]
{
\newpath
\moveto(131.63648934,461.95821629)
\curveto(280.32310591,515.04191229)(440.89135245,542.09075299)(606.39771967,541.93505584)
}
}
{
\newrgbcolor{curcolor}{0 0 0}
\pscustom[linewidth=0.09830981,linecolor=curcolor]
{
\newpath
\moveto(367.58365,601.68913262)
\lineto(367.15995,559.80566262)
}
}
{
\newrgbcolor{curcolor}{0 0 0}
\pscustom[linewidth=0.09830981,linecolor=curcolor]
{
\newpath
\moveto(373.09154,561.08646262)
\lineto(381.3534,603.90429262)
}
}
{
\newrgbcolor{curcolor}{0 0 0}
\pscustom[linewidth=0.09830981,linecolor=curcolor]
{
\newpath
\moveto(379.65865,558.15505262)
\lineto(394.27577,596.83851262)
}
}
{
\newrgbcolor{curcolor}{0 0 0}
\pscustom[linewidth=0.09830981,linecolor=curcolor]
{
\newpath
\moveto(385.37839,550.41187262)
\curveto(385.44901,550.44042262)(385.51963,550.46896262)(385.37839,550.41187262)
\closepath
}
}
{
\newrgbcolor{curcolor}{0 0 0}
\pscustom[linewidth=0.09830981,linecolor=curcolor]
{
\newpath
\moveto(385.06066,551.58692262)
\lineto(403.49091,585.10826262)
}
}
{
\newrgbcolor{curcolor}{0 0 0}
\pscustom[linewidth=0.09830981,linecolor=curcolor]
{
\newpath
\moveto(388.02643,543.10277262)
\lineto(410.05802,567.09341262)
}
}
{
\newrgbcolor{curcolor}{0 0 0}
\pscustom[linewidth=0.09830981,linecolor=curcolor]
{
\newpath
\moveto(389.29747,533.93346262)
\lineto(413.02379,545.94663262)
}
}
{
\newrgbcolor{curcolor}{0 0 0}
\pscustom[linewidth=0.09830981,linecolor=curcolor]
{
\newpath
\moveto(389.19155,524.57988262)
\lineto(412.28237,525.53565262)
}
}
{
\newrgbcolor{curcolor}{0 0 0}
\pscustom[linewidth=0.09830981,linecolor=curcolor]
{
\newpath
\moveto(387.70867,513.92473262)
\lineto(409.10472,505.44321262)
}
}
{
\newrgbcolor{curcolor}{0 0 0}
\pscustom[linewidth=0.09830981,linecolor=curcolor]
{
\newpath
\moveto(384.74286,503.78736262)
\lineto(402.96131,486.38633262)
}
}
{
\newrgbcolor{curcolor}{0 0 0}
\pscustom[linewidth=0.09830981,linecolor=curcolor]
{
\newpath
\moveto(380.71788,495.45633262)
\curveto(380.75317,495.47060262)(380.78847,495.48488262)(380.71788,495.45633262)
\closepath
}
}
{
\newrgbcolor{curcolor}{0 0 0}
\pscustom[linewidth=0.09830981,linecolor=curcolor]
{
\newpath
\moveto(380.18827,494.49736262)
\lineto(394.16985,467.37614262)
}
}
{
\newrgbcolor{curcolor}{0 0 0}
\pscustom[linewidth=0.09830981,linecolor=curcolor]
{
\newpath
\moveto(373.93892,486.19811262)
\lineto(381.67116,450.77767262)
}
}
{
\newrgbcolor{curcolor}{0 0 0}
\pscustom[linewidth=0.09830981,linecolor=curcolor]
{
\newpath
\moveto(368.64285,482.38103262)
\curveto(368.60753,482.36676262)(368.57223,482.35248262)(368.64285,482.38103262)
\closepath
}
}
{
\newrgbcolor{curcolor}{0 0 0}
\pscustom[linewidth=0.09830981,linecolor=curcolor]
{
\newpath
\moveto(368.21918,482.02353262)
\lineto(369.49023,442.12873262)
}
}
{
\newrgbcolor{curcolor}{0 0 0}
\pscustom[linewidth=0.09830981,linecolor=curcolor]
{
\newpath
\moveto(361.75799,480.71482262)
\lineto(355.08495,439.47042262)
}
}
{
\newrgbcolor{curcolor}{0 0 0}
\pscustom[linewidth=0.09830981,linecolor=curcolor]
{
\newpath
\moveto(355.93233,482.27006262)
\lineto(343.00994,443.71312262)
}
}
{
\newrgbcolor{curcolor}{0 0 0}
\pscustom[linewidth=0.09830981,linecolor=curcolor]
{
\newpath
\moveto(350.31851,488.19389262)
\curveto(350.2479,488.16534262)(350.1773,488.13680262)(350.31851,488.19389262)
\closepath
}
}
{
\newrgbcolor{curcolor}{0 0 0}
\pscustom[linewidth=0.09830981,linecolor=curcolor]
{
\newpath
\moveto(350.31851,488.38012262)
\curveto(350.1773,488.32300262)(350.1773,488.32300262)(350.31851,488.38012262)
\closepath
}
}
{
\newrgbcolor{curcolor}{0 0 0}
\pscustom[linewidth=0.09830981,linecolor=curcolor]
{
\newpath
\moveto(350.21259,488.15108262)
\lineto(332.52376,454.92949262)
}
}
{
\newrgbcolor{curcolor}{0 0 0}
\pscustom[linewidth=0.09830981,linecolor=curcolor]
{
\newpath
\moveto(346.8231,496.27771262)
\curveto(346.68187,496.22060262)(346.68187,496.22060262)(346.8231,496.27771262)
\closepath
}
}
{
\newrgbcolor{curcolor}{0 0 0}
\pscustom[linewidth=0.09830981,linecolor=curcolor]
{
\newpath
\moveto(346.8231,496.27771262)
\lineto(326.5922,468.91832262)
}
}
{
\newrgbcolor{curcolor}{0 0 0}
\pscustom[linewidth=0.09830981,linecolor=curcolor]
{
\newpath
\moveto(345.02245,506.53641262)
\lineto(322.35534,489.17879262)
}
}
{
\newrgbcolor{curcolor}{0 0 0}
\pscustom[linewidth=0.09830981,linecolor=curcolor]
{
\newpath
\moveto(345.02245,516.77824262)
\curveto(345.02245,517.02653262)(345.02245,517.02653262)(345.02245,516.77824262)
\closepath
}
}
{
\newrgbcolor{curcolor}{0 0 0}
\pscustom[linewidth=0.09830981,linecolor=curcolor]
{
\newpath
\moveto(345.23429,518.91225262)
\curveto(345.19899,518.89798262)(345.16367,518.88370262)(345.23429,518.91225262)
\closepath
}
}
{
\newrgbcolor{curcolor}{0 0 0}
\pscustom[linewidth=0.09830981,linecolor=curcolor]
{
\newpath
\moveto(345.02245,518.26796262)
\lineto(321.93165,513.02925262)
}
}
{
\newrgbcolor{curcolor}{0 0 0}
\pscustom[linewidth=0.09830981,linecolor=curcolor]
{
\newpath
\moveto(346.92904,529.28062262)
\curveto(346.71932,529.56451262)(346.8563,529.37907262)(346.92904,529.28062262)
\closepath
}
}
{
\newrgbcolor{curcolor}{0 0 0}
\pscustom[linewidth=0.09830981,linecolor=curcolor]
{
\newpath
\moveto(346.92904,529.28062262)
\lineto(325.53295,536.64483262)
}
}
{
\newrgbcolor{curcolor}{0 0 0}
\pscustom[linewidth=0.09830981,linecolor=curcolor]
{
\newpath
\moveto(355.61456,549.73769262)
\curveto(355.61456,549.79975262)(355.61456,549.86184262)(355.61456,549.73769262)
\closepath
}
}
{
\newrgbcolor{curcolor}{0 0 0}
\pscustom[linewidth=0.09830981,linecolor=curcolor]
{
\newpath
\moveto(355.61456,549.55147262)
\lineto(342.05666,578.51991262)
}
}
{
\newrgbcolor{curcolor}{0 0 0}
\pscustom[linewidth=0.09830981,linecolor=curcolor]
{
\newpath
\moveto(361.3343,556.14689262)
\lineto(354.97902,593.98617262)
}
}
{
\newrgbcolor{curcolor}{0 0 0}
\pscustom[linewidth=0.09830981,linecolor=curcolor]
{
\newpath
\moveto(350.31851,539.77550262)
\curveto(350.1773,539.71839262)(350.1773,539.71839262)(350.31851,539.77550262)
\closepath
}
}
{
\newrgbcolor{curcolor}{0 0 0}
\pscustom[linewidth=0.09830981,linecolor=curcolor]
{
\newpath
\moveto(350.31851,539.58928262)
\lineto(331.99416,558.43721262)
}
}
{
\newrgbcolor{curcolor}{0 0 0}
\pscustom[linewidth=1.99999999,linecolor=curcolor]
{
\newpath
\moveto(424.48924265,469.29684918)
\curveto(400.28560309,404.49503268)(355.24410621,376.16838893)(323.95018711,406.06773367)
\curveto(292.65626801,435.9670784)(286.90183824,512.82599187)(311.10547781,577.62780837)
\curveto(335.30911737,642.42962487)(380.35061425,670.75626862)(411.64453335,640.85692388)
\curveto(423.40863774,629.61707209)(431.07815143,613.71955886)(435.78422735,590.81970894)
}
}
{
\newrgbcolor{curcolor}{0 0 0}
\pscustom[linewidth=0.2700935,linecolor=curcolor]
{
\newpath
\moveto(402.39930328,587.45466485)
\curveto(418.65180206,560.03358383)(416.15015192,508.24501509)(396.81526278,471.85518107)
\curveto(395.57498025,469.52086845)(394.48135025,467.62128974)(393.14248596,465.47575781)
}
}
{
\newrgbcolor{curcolor}{0 0 0}
\pscustom[linewidth=1.99999993,linecolor=curcolor]
{
\newpath
\moveto(392.2319195,464.18016395)
\curveto(371.02218236,431.49165066)(342.67665512,430.72400595)(328.96065335,462.46667129)
\curveto(315.24465159,494.20933664)(321.32652378,546.50130771)(342.53626092,579.18982101)
\curveto(362.22600804,609.53572091)(387.90883077,612.86066644)(402.77309667,586.9881927)
}
}
{
\newrgbcolor{curcolor}{0 0 0}
\pscustom[linewidth=2.00000005,linecolor=curcolor]
{
\newpath
\moveto(360.19138885,459.29449766)
\curveto(341.83383829,458.77531723)(330.12249975,485.94866647)(334.05000819,519.94929508)
\curveto(337.5458721,550.21315639)(351.90534673,575.75409158)(368.60297957,581.40803401)
}
}
{
\newrgbcolor{curcolor}{0 0 0}
\pscustom[linewidth=0.27009409,linecolor=curcolor]
{
\newpath
\moveto(367.58503545,581.188626)
\curveto(386.35954071,588.52186036)(401.4780979,567.63747699)(401.33181489,534.57169882)
\curveto(401.18553188,501.50592066)(385.82952981,468.71831629)(367.05502455,461.38508193)
\curveto(365.30556661,460.70175181)(363.82604118,460.31467917)(362.09727846,460.0880391)
}
}
{
\newrgbcolor{curcolor}{0 0 0}
\pscustom[linewidth=0.06104613,linecolor=curcolor]
{
\newpath
\moveto(367.40635765,560.02010697)
\curveto(379.69740487,564.75000788)(389.56228387,551.00686342)(389.42617684,529.34344522)
\curveto(389.29302589,508.15053106)(379.69449434,487.19192287)(367.67284927,481.84456825)
}
}
{
\newrgbcolor{curcolor}{0 0 0}
\pscustom[linewidth=0.06104613,linecolor=curcolor]
{
\newpath
\moveto(367.42535746,481.88578984)
\curveto(355.13435432,476.67043858)(345.04539857,489.92892429)(344.90536692,511.48062403)
\curveto(344.76533528,533.03232377)(354.62699331,554.7563072)(366.91799644,559.97165847)
\curveto(367.02247761,560.0159922)(367.11083885,560.05244867)(367.21532661,560.09433226)
}
}
{
\newrgbcolor{curcolor}{0 0 0}
\pscustom[linewidth=0.09830981,linecolor=curcolor]
{
\newpath
\moveto(367.59546,601.84314262)
\lineto(367.17176,559.95967262)
}
}
{
\newrgbcolor{curcolor}{0 0 0}
\pscustom[linewidth=0.09830981,linecolor=curcolor]
{
\newpath
\moveto(373.10334,561.24045262)
\lineto(381.36519,604.05830262)
}
}
{
\newrgbcolor{curcolor}{0 0 0}
\pscustom[linewidth=0.09830981,linecolor=curcolor]
{
\newpath
\moveto(379.67044,558.30905262)
\lineto(394.28757,596.99251262)
}
}
{
\newrgbcolor{curcolor}{0 0 0}
\pscustom[linewidth=0.09830981,linecolor=curcolor]
{
\newpath
\moveto(385.39019,550.56588262)
\curveto(385.46081,550.59443262)(385.53143,550.62297262)(385.39019,550.56588262)
\closepath
}
}
{
\newrgbcolor{curcolor}{0 0 0}
\pscustom[linewidth=0.09830981,linecolor=curcolor]
{
\newpath
\moveto(385.07244,551.74092262)
\lineto(403.50271,585.26228262)
}
}
{
\newrgbcolor{curcolor}{0 0 0}
\pscustom[linewidth=0.09830981,linecolor=curcolor]
{
\newpath
\moveto(388.03823,543.25677262)
\lineto(410.06982,567.24742262)
}
}
{
\newrgbcolor{curcolor}{0 0 0}
\pscustom[linewidth=0.09830981,linecolor=curcolor]
{
\newpath
\moveto(389.30927,534.08747262)
\lineto(413.0356,546.10064262)
}
}
{
\newrgbcolor{curcolor}{0 0 0}
\pscustom[linewidth=0.09830981,linecolor=curcolor]
{
\newpath
\moveto(389.20335,524.73390262)
\lineto(412.29417,525.68966262)
}
}
{
\newrgbcolor{curcolor}{0 0 0}
\pscustom[linewidth=0.09830981,linecolor=curcolor]
{
\newpath
\moveto(387.72046,514.07874262)
\lineto(409.11654,505.59722262)
}
}
{
\newrgbcolor{curcolor}{0 0 0}
\pscustom[linewidth=0.09830981,linecolor=curcolor]
{
\newpath
\moveto(384.75467,503.94137262)
\lineto(402.97312,486.54033262)
}
}
{
\newrgbcolor{curcolor}{0 0 0}
\pscustom[linewidth=0.09830981,linecolor=curcolor]
{
\newpath
\moveto(380.72968,495.61035262)
\curveto(380.76496,495.62460262)(380.80029,495.63888262)(380.72968,495.61035262)
\closepath
}
}
{
\newrgbcolor{curcolor}{0 0 0}
\pscustom[linewidth=0.09830981,linecolor=curcolor]
{
\newpath
\moveto(380.20006,494.65137262)
\lineto(394.18166,467.53015262)
}
}
{
\newrgbcolor{curcolor}{0 0 0}
\pscustom[linewidth=0.09830981,linecolor=curcolor]
{
\newpath
\moveto(373.95072,486.35211262)
\lineto(381.68295,450.93167262)
}
}
{
\newrgbcolor{curcolor}{0 0 0}
\pscustom[linewidth=0.09830981,linecolor=curcolor]
{
\newpath
\moveto(368.65466,482.53504262)
\curveto(368.61935,482.52075262)(368.58404,482.50647262)(368.65466,482.53504262)
\closepath
}
}
{
\newrgbcolor{curcolor}{0 0 0}
\pscustom[linewidth=0.09830981,linecolor=curcolor]
{
\newpath
\moveto(368.23098,482.17753262)
\lineto(369.50204,442.28272262)
}
}
{
\newrgbcolor{curcolor}{0 0 0}
\pscustom[linewidth=0.09830981,linecolor=curcolor]
{
\newpath
\moveto(361.76978,480.86881262)
\lineto(355.09675,439.62442262)
}
}
{
\newrgbcolor{curcolor}{0 0 0}
\pscustom[linewidth=0.09830981,linecolor=curcolor]
{
\newpath
\moveto(355.94412,482.42405262)
\lineto(343.02175,443.86712262)
}
}
{
\newrgbcolor{curcolor}{0 0 0}
\pscustom[linewidth=0.09830981,linecolor=curcolor]
{
\newpath
\moveto(350.33031,488.34791262)
\curveto(350.2597,488.31936262)(350.18908,488.29079262)(350.33031,488.34791262)
\closepath
}
}
{
\newrgbcolor{curcolor}{0 0 0}
\pscustom[linewidth=0.09830981,linecolor=curcolor]
{
\newpath
\moveto(350.33031,488.53411262)
\curveto(350.18908,488.47702262)(350.18908,488.47702262)(350.33031,488.53411262)
\closepath
}
}
{
\newrgbcolor{curcolor}{0 0 0}
\pscustom[linewidth=2.00000008,linecolor=curcolor]
{
\newpath
\moveto(368.06806705,581.02071014)
\curveto(452.55090376,594.96541578)(525.99061335,601.08898196)(607.29017883,600.9675814)
}
}
{
\newrgbcolor{curcolor}{0 0 0}
\pscustom[linewidth=0.09830981,linecolor=curcolor]
{
\newpath
\moveto(350.2244,488.30509262)
\lineto(332.53557,455.08351262)
}
}
{
\newrgbcolor{curcolor}{0 0 0}
\pscustom[linewidth=0.09830981,linecolor=curcolor]
{
\newpath
\moveto(346.83492,496.43171262)
\curveto(346.69368,496.37460262)(346.69368,496.37460262)(346.83492,496.43171262)
\closepath
}
}
{
\newrgbcolor{curcolor}{0 0 0}
\pscustom[linewidth=0.09830981,linecolor=curcolor]
{
\newpath
\moveto(346.83492,496.43171262)
\lineto(326.60399,469.07233262)
}
}
{
\newrgbcolor{curcolor}{0 0 0}
\pscustom[linewidth=0.09830981,linecolor=curcolor]
{
\newpath
\moveto(345.03424,506.69041262)
\lineto(322.36713,489.33277262)
}
}
{
\newrgbcolor{curcolor}{0 0 0}
\pscustom[linewidth=0.09830981,linecolor=curcolor]
{
\newpath
\moveto(345.03424,516.93225262)
\curveto(345.03424,517.18054262)(345.03424,517.18054262)(345.03424,516.93225262)
\closepath
}
}
{
\newrgbcolor{curcolor}{0 0 0}
\pscustom[linewidth=0.09830981,linecolor=curcolor]
{
\newpath
\moveto(345.24609,519.06626262)
\curveto(345.21077,519.05199262)(345.17548,519.03771262)(345.24609,519.06626262)
\closepath
}
}
{
\newrgbcolor{curcolor}{0 0 0}
\pscustom[linewidth=0.09830981,linecolor=curcolor]
{
\newpath
\moveto(345.03424,518.42197262)
\lineto(321.94345,513.18326262)
}
}
{
\newrgbcolor{curcolor}{0 0 0}
\pscustom[linewidth=0.09830981,linecolor=curcolor]
{
\newpath
\moveto(346.94082,529.43462262)
\curveto(346.7311,529.71852262)(346.8681,529.53307262)(346.94082,529.43462262)
\closepath
}
}
{
\newrgbcolor{curcolor}{0 0 0}
\pscustom[linewidth=0.09830981,linecolor=curcolor]
{
\newpath
\moveto(346.94082,529.43462262)
\lineto(325.54476,536.79884262)
}
}
{
\newrgbcolor{curcolor}{0 0 0}
\pscustom[linewidth=2.00000008,linecolor=curcolor]
{
\newpath
\moveto(359.89185922,458.90639363)
\curveto(446.81422707,473.74224293)(522.6700746,480.30050056)(606.49272475,480.22672767)
}
}
{
\newrgbcolor{curcolor}{0 0 0}
\pscustom[linewidth=0.09830981,linecolor=curcolor]
{
\newpath
\moveto(355.62636,549.89170262)
\curveto(355.62636,549.95376262)(355.62636,550.01584262)(355.62636,549.89170262)
\closepath
}
}
{
\newrgbcolor{curcolor}{0 0 0}
\pscustom[linewidth=0.09830981,linecolor=curcolor]
{
\newpath
\moveto(355.62636,549.70548262)
\lineto(342.06846,578.67393262)
}
}
{
\newrgbcolor{curcolor}{0 0 0}
\pscustom[linewidth=0.09830981,linecolor=curcolor]
{
\newpath
\moveto(361.3461,556.30087262)
\lineto(354.99084,594.14018262)
}
}
{
\newrgbcolor{curcolor}{0 0 0}
\pscustom[linewidth=0.09830981,linecolor=curcolor]
{
\newpath
\moveto(350.33031,539.92950262)
\curveto(350.18908,539.87240262)(350.18908,539.87240262)(350.33031,539.92950262)
\closepath
}
}
{
\newrgbcolor{curcolor}{0 0 0}
\pscustom[linewidth=0.09830981,linecolor=curcolor]
{
\newpath
\moveto(350.33031,539.74328262)
\lineto(332.00595,558.59122262)
}
}
{
\newrgbcolor{curcolor}{0 0 0}
\pscustom[linewidth=0.27009347,linecolor=curcolor]
{
\newpath
\moveto(435.90903494,590.41871464)
\curveto(443.82393425,551.19610043)(439.21332466,506.54239693)(423.02137063,465.60182664)
}
}
{
\newrgbcolor{curcolor}{0 0 0}
\pscustom[linewidth=1.99999994,linecolor=curcolor]
{
\newpath
\moveto(238.87630475,433.83327431)
\curveto(220.24045075,430.1859855)(206.92867602,454.85337506)(209.16252189,488.89441171)
\curveto(211.25126335,520.72423746)(225.96268173,549.54186697)(243.51012748,556.17661027)
}
}
{
\newrgbcolor{curcolor}{0 0 0}
\pscustom[linewidth=0.27000001,linecolor=curcolor]
{
\newpath
\moveto(243.26876768,555.84572791)
\curveto(262.04277627,563.23190396)(277.16093351,542.19674215)(277.01465437,508.89224172)
\curveto(276.86837523,475.58774129)(261.51277939,442.56342414)(242.7387708,435.17724809)
\curveto(241.89746131,434.84625538)(241.18610004,434.61075587)(240.34733504,434.38555198)
}
}
{
\newrgbcolor{curcolor}{0 0 0}
\pscustom[linewidth=1.99999998,linecolor=curcolor]
{
\newpath
\moveto(532.90477649,478.66778329)
\curveto(514.55367884,478.2424099)(502.88628013,505.05050928)(506.8614758,538.50729083)
\curveto(510.80286184,571.67951754)(528.61683511,599.00742173)(546.862885,599.87230252)
}
}
{
\newrgbcolor{curcolor}{0 0 0}
\pscustom[linewidth=0.27009387,linecolor=curcolor]
{
\newpath
\moveto(550.26014545,599.04191635)
\curveto(569.03465071,606.31260698)(584.1532079,585.60639267)(584.00692489,552.82268799)
\curveto(583.86064188,520.03898331)(568.50463981,487.53106828)(549.73013455,480.26037766)
\curveto(545.85592665,478.76003629)(542.57419858,478.42567804)(538.91135582,479.15810842)
}
}
{
\newrgbcolor{curcolor}{0 0 0}
\pscustom[linewidth=0.27009416,linecolor=curcolor]
{
\newpath
\moveto(450.4188,594.14147262)
\lineto(450.41878,472.13808262)
}
}
{
\newrgbcolor{curcolor}{0 0 0}
\pscustom[linewidth=2,linecolor=curcolor]
{
\newpath
\moveto(284.3134,634.34908262)
\lineto(284.31339,345.40335262)
}
}
{
\newrgbcolor{curcolor}{0 0 0}
\pscustom[linewidth=2,linecolor=curcolor]
{
\newpath
\moveto(450.58372,702.25344262)
\lineto(450.58368,592.63564262)
}
}
{
\newrgbcolor{curcolor}{0 0 0}
\pscustom[linewidth=2,linecolor=curcolor]
{
\newpath
\moveto(283.83228,634.52559262)
\lineto(450.99344,702.54100262)
}
}
{
\newrgbcolor{curcolor}{0 0 0}
\pscustom[linewidth=2,linecolor=curcolor]
{
\newpath
\moveto(283.86003,344.88593262)
\lineto(450.77936,412.67287262)
}
}
{
\newrgbcolor{curcolor}{0 0 0}
\pscustom[linewidth=2,linecolor=curcolor]
{
\newpath
\moveto(450.31168,471.28818262)
\lineto(450.31161,411.62550262)
}
}
{
\newrgbcolor{curcolor}{0 0 0}
\pscustom[linewidth=2.00000008,linecolor=curcolor]
{
\newpath
\moveto(131.56730363,522.12244716)
\curveto(181.62869089,540.04743336)(227.63793895,553.2859556)(284.02363059,565.98963821)
}
}
{
\newrgbcolor{curcolor}{0 0 0}
\pscustom[linewidth=2.00000008,linecolor=curcolor]
{
\newpath
\moveto(132.47241428,400.90309601)
\curveto(182.35203751,418.71506022)(228.16203751,431.87214547)(284.2902585,444.50668296)
}
}
{
\newrgbcolor{curcolor}{0 0 0}
\pscustom[linewidth=0.27009373,linecolor=curcolor]
{
\newpath
\moveto(284.30238653,444.50941294)
\curveto(312.21585633,450.79252038)(336.26706373,455.5072057)(364.92962439,460.31452609)
}
}
{
\newrgbcolor{curcolor}{0 0 0}
\pscustom[linewidth=0.27009373,linecolor=curcolor]
{
\newpath
\moveto(283.33784575,565.50813013)
\curveto(312.29868284,572.02008602)(337.26919347,576.88250886)(367.02274529,581.8038068)
}
}
{
\newrgbcolor{curcolor}{0 0 0}
\pscustom[linewidth=2,linecolor=curcolor]
{
\newpath
\moveto(470.14173187,474.15873659)
\curveto(451.95415072,474.99093716)(440.97522149,502.91152476)(445.63519342,536.48147675)
\curveto(450.14428741,568.96451986)(467.53460264,594.93989738)(485.33978757,595.78700452)
}
}
{
\newrgbcolor{curcolor}{0 0 0}
\pscustom[linewidth=0.27,linecolor=curcolor]
{
\newpath
\moveto(483.07094735,594.64443834)
\curveto(501.84464219,601.99508093)(516.96254678,581.06118217)(516.81627009,547.91698576)
\curveto(516.67027012,514.83549198)(501.38055687,482.03161764)(482.64202722,474.5966069)
}
}
{
\newrgbcolor{curcolor}{0 0 0}
\pscustom[linewidth=0.27000001,linecolor=curcolor]
{
\newpath
\moveto(185.84279155,540.22217063)
\curveto(204.61662606,547.62203456)(219.73464312,526.54794285)(219.58836534,493.18178736)
\curveto(219.44208756,459.81563187)(204.0866341,426.73016685)(185.31279958,419.33030292)
\curveto(183.46526199,418.60208053)(181.90272671,418.20102186)(180.07943572,417.98704627)
}
}
{
\newrgbcolor{curcolor}{0 0 0}
\pscustom[linewidth=2.00000006,linecolor=curcolor]
{
\newpath
\moveto(180.93375187,417.37634029)
\curveto(162.3584014,414.54707503)(149.46055137,439.9968907)(152.14390804,474.18396456)
\curveto(154.60264129,505.50924297)(169.00617376,533.35871027)(186.202026,540.03589667)
}
}
{
\newrgbcolor{curcolor}{0 0 0}
\pscustom[linewidth=1.35000002,linecolor=curcolor]
{
\newpath
\moveto(433.04148,657.87940262)
\lineto(493.16226,684.27575262)
}
}
{
\newrgbcolor{curcolor}{0 0 0}
\pscustom[linewidth=1.35046935,linecolor=curcolor]
{
\newpath
\moveto(390.79005,488.15586262)
\lineto(490.01283,421.75865262)
}
}
{
\newrgbcolor{curcolor}{0 0 0}
\pscustom[linewidth=1.35000002,linecolor=curcolor]
{
\newpath
\moveto(211.54218,591.93673262)
\lineto(260.05691,500.40774262)
}
}
{
\newrgbcolor{curcolor}{0 0 0}
\pscustom[linewidth=1.35046935,linecolor=curcolor]
{
\newpath
\moveto(369.33461,452.04915262)
\lineto(488.53569,373.66497262)
}
}
{
\newrgbcolor{curcolor}{0 0 0}
\pscustom[linewidth=0.27009389,linecolor=curcolor]
{
\newpath
\moveto(367.10951,526.02133262)
\lineto(367.10951,517.58090262)
}
}
\rput(563,688){$B_m\x\R^{m+k}$}
\rput(523,372){$-\tau_m$}
\rput(515,418){$E_m$}
\rput(205,610){$B_m$}

\end{pspicture}

\caption{The \emph{umkehr map}}\label{aua}\end{figure}

The homotopy classes of $t_m$ are independent of the chosen sections $B_m\to\widetilde{E}_m\x_G \Emb(F,\R^{m+k})^{\text{fad}}$ because two such sections are homotopic, $\Emb(F,\R^{m+k})^{\text{fad}}$ being $(m-k-2)$-connected.

\begin{rem}
We do not want to go into the question whether the homotopy class of $t$ is independent of the chosen sections.
\end{rem}

Before defining the transfer we recall the Thom isomorphism with respect to a stable vector bundle and an arbitrary homology theory. We remind the reader that $X$ denotes a ring spectrum.

To a stable vector bundle $\eta_n\to A_n$ one can associate a spectrum $M(\eta)$ in a natural way. Simply put $(M(\eta))_n=Th(\eta_n)$. Then the given isomorphisms $\eta_n\oplus\R\cong\eta_{n+1}$ provide the necessary structure maps after taking the Thom space functor. The inclusions of fibers $\R^n\hto\eta_n$ induce maps $S^n\cong D^n/S^{n-1}\to\Th(\eta_n)$ and so maps of spectra $i\colon\Sigma^{\infty}(S^0)\to M(\eta)$. An element $u\in X^0(M(\eta))$ is called a Thom class for $\eta$ with respect to $X$ if $i^*(u)\in X^0(\Sigma^{\infty}(S^0))\cong\pi_*(X)$ is a generator of $\pi_*(X)$ for all fibers. The bundle $\eta$ is called orientable with respect to $X$ if a Thom class exists.

Now let $\eta\to A$ be a stable vector bundle which is oriented with respect to $X$ with Thom class $u\colon M(\eta)\to X$. Further let
\[
M(\Delta_*)\colon M(\eta)\to M(\eta)\wedge A_+
\]
induced by
\[
\xgrid=16mm
\ygrid=7mm
\cellpush=6pt
\flexible
\Diagram
\eta_n={} & \Delta^*(\eta_n\x A) &  \rTo^{\Delta_*} &  \eta_n\x A \\
          & \dTo                 &                  &  \dTo     \\
          & A                    &  \rTo^{\Delta}   &  A\x A   \\
\endDiagram
\]
where $\Delta\colon A\to A\x A$ denotes the diagonal map.
The Thom isomorphism theorem (cf. \cite{switzer}, Theorem 14.6) states that the map
\begin{equation*}
\widetilde{\Psi}\colon X\wedge M(\eta)  \xto{\id\wedge M(\Delta_*)}  X\wedge M(\eta)\wedge A_+
                                        \xto{\id\wedge u\wedge\id} X\wedge X      \wedge A_+
                                        \xto{\mu\wedge\id}         X\wedge A_+,
\end{equation*}
where $\mu\colon X\wedge X\to X$ denotes the spectrum multiplication, induces on homotopy groups isomorphisms
\[
\Psi\colon X_i(M(\eta))\cong X_i(A_+)
\]
for all $i$. Similarly one gets isomorphisms $X^i(A_+)\cong X^i(M(\eta))$.

Now we are able to define the bundle transfer. Note that a $X$-orientation of $\tau_m$ induces a $X$-orientation of $-\tau_m$ (cf. \cite{rud}, Ch. 5, Proposition 1.10).

\begin{dfn} The bundle transfer is given by
\begin{equation}\label{tinh}
X_{n-k}(B_+) \xto{\sigma} X_n(\Sigma^kB_+) \xto{t_*} X_n(M(-\tau)) \xto{\Psi} X_n(E_+)
\end{equation}
resp.
\begin{equation}\label{tink}
X^n(E_+) \xto{\Psi} X^n(M(-\tau)) \xto{t^*} X^n(\Sigma^kB_+) \xto{\sigma} X^{n-k}(B_+),
\end{equation}
where $\sigma$ denotes the suspension isomorphism and $\Psi$ the Thom isomorphism (both for homology and cohomology).
\end{dfn}

Note that the bundle transfer is natural in the following sense.
Consider two fibre bundles $F\hto E\xto{p} B$ and $F'\hto E'\xto{p'} B'$ which tangent bundles along the fibers $\tau$ and $\tau'$ possess Thom classes $u$ and $u'$.
Let further
\[
\xgrid=7mm
\ygrid=7mm
\cellpush=6pt
\Diagram
E        &  \rTo^{\Phi}   &  E'  \\
\dTo^{p} &             &  \dTo^{p'}  \\
B        &  \rTo^{\phi}   &  B'       \\
\endDiagram
\]
be a map of fibre bundles such that $\Phi$ restricted to the fibres induces an isomorphism $X_*(F)\xto{\cong} X_*(F')$ and the induced map $X^0(M(-\tau'))\to  X^0(M(-\tau))$ maps $u'$ to $u$.
Then the diagram
\[
\xgrid=15mm
\ygrid=7mm
\cellpush=6pt
\Diagram
X_n(E_+)        &  \rTo^{\Phi_*}   &  X_n(E'_+)      \\
\uTo_{p^!}      &               &  \uTo_{p'^!}    \\
X_{n-k}(B_+)    &  \rTo^{\phi_*}   &  X_{n-k}(B'_+)  \\
\endDiagram
\]
commutes.

In the two following propositions we shall interpret this homotopy theoretic definition of the bundle transfer in more geometric terms. Recall that in bordism theory one often considers fibrations $BG_n\xto{h} BO(n)$ induced by $G_n\to O(n)$,
$G_n$ a 'classical' Lie group (cf. \cite{switzer}, Ch. 12). By definition a vector bundle admits a '$G$-structure'
if one can lift its classifying map to $BG_n$. The pullbacks $\gc_n^G:=h^*(\gc_n)\to BG_n$ of the universal bundles $\gc_n\to BO(n)$ form a stable vector bundle which associated Thom spectrum is denoted by $MG$. It is not difficult to show that a vector bundle with $G$-structure is oriented with respect to $MG$ (the reverse direction is false in general). A manifold is called a $G$-manifold if its normal bundle admits a $G$-structure (in most cases of interest, e.g. $G_n=SO(n), Spin(n)$, this equivalent to say that its tangent bundle admits a $G$-structure).
\begin{prop}\label{hint}
For $X=MG$ the transfer \ref{tinh}, i.e. $\Omega_{n-k}^G(B)\to\Omega_{n}^G(E)$, coincides with assigning to a $G$-manifold $f\colon N\to B$ its pullback $\hat{f}\colon f^*(E)\to E$.
\end{prop}
Note that $\hat{f}\colon f^*(E)\to E$ depends only on the cobordism class of $f\colon N\to B$ and that $f^*(E)$ is again a $G$-manifold as $T(f^*E)\cong TN\oplus T(\hat{f}^*\tau)$, i.e. the pullback-map is well defined.
\begin{proof}
Any $[N,f]\in\Omega_n^G(B)$ can be represented by a map $\phi\colon S^{n+l}\to\Th(\gc_l^G)\wedge B_+$ with $\phi^{-1}(BG_l\x B)=N$.
Since $S^{n+l}$ is compact $\im\phi$ is contained is some $\Th(\gc_l^G)\wedge (B_m)_+$. Let $v_m\colon\Th(-\tau_m)\to\Th(\gc_m^G)$ be the $MG$-Thom class of $-\tau_m$. The image of $[\phi]$ under the transfer \ref{tinh} is given by
\begin{equation}\label{transfer}
\begin{split}
\hat{\phi}\colon S^{n+l}\wedge S^{k+m}&\xto{\phi\wedge\id}                \Th(\gc_l^G)\wedge (B_m)_+\wedge S^{k+m}\\
                                      &\xto{\id\wedge t_m}                \Th(\gc_l^G)\wedge\Th(-\tau_m)\\
                                      &\xto{\id\wedge\Th(\Delta_*)}       \Th(\gc_l^G)\wedge\Th(-\tau_m)\wedge (E_m)_+\\
                                      &\xto{\id\wedge(v_m)\wedge\id}      \Th(\gc_l^G)\wedge\Th(\gc_m^G)\wedge (E_m)_+\\
                                      &\xto{\mu\wedge\id}                 \Th(\gc_{l+m}^G)\wedge (E_m)_+,
\end{split}
\end{equation}
where the rows 2 to 5 contain the Thom isomorphism.

We have to show that $\hat{\phi}^{-1}(BG_{l+m}\x E_m)=f^*(E_m)$. The preimage of $BG_{l+m}\x E_m$ under the Thom isomorphism is the zero section $BG_l\x E_m\hto\gc_l^G\x(-\tau_m)$ as all involved maps are induced by maps of vector bundles (which means fiberwise isomorphisms).

The preimage of $E_m\hto\Th(-\tau_m)$ under $t_m$ is our embedding $E_m\subset B_m\x \R^{k+m}$, cf. \ref{emb}. The preimage of $BG_l\x\{b\}\x p^{-1}(b)$ under $\phi\wedge\id$ consists of all $(n,e)\in\R^{n+l}\times\R^{k+m}$ with $n\in N$ and $e\in E_m$ such that $f(n)=b$ and $p(e)=b$. This is nothing else than the pullback $f^*E_m$. Note that $\hat{\phi}|_{f^*(E_m)}$ composed with the projection onto $BG_{l+m}$ yields a $G$-structure on $f^*(E_m)$.
\end{proof}


Let $R$ be a ring. Now we treat the bundle transfer for $X=HR$, the Eilenberg-MacLane spectrum of $R$. Let $(F,F')\hto (E,E')\xto{p} B$ be fiber bundle of CW-pairs. In addition let $(F,F')$ have cohomological dimension $k$, i.e. $H^k(F,F';R)\cong R$, $H^{k'}(F,F';R)=0$ for $k'>k$ and let $\pi_1(B)$ act trivially on $H^k(F,F';R)$. The \emph{integration along the fiber} is defined by means of the Leray-Serre spectral sequence:
\begin{align*}
p_!\colon H^n(E,E';R)\twoheadrightarrow E_{\infty}^{n-k,k}\rightarrowtail E_2^{n-k,k}&\cong H^{n-k}(B;H^k(F,F';R))\\
                                                                                              &\cong H^{n-k}(B;R).
\end{align*}

\begin{prop}\label{gint}
For $X=HR$ the bundle transfer \ref{tink} coincides with the integration along the fiber.
\end{prop}
Note that the an orientation class of $\tau$ induces orientation classes of $T(p^{-1}(b))$, varying continuously in $b\in B$, which induce in turn fundamental classes $[p^{-1}(b)]\in H^k(p^{-1}(b);R)$. It follows that $H^k(p^{-1}(b);R)\cong R$ and that $\pi_1(B)$ acts trivially on $H^k(p^{-1}(b);R)$ as $[p^{-1}(b)]$ is a generator of $H^k(p^{-1}(b);R)$, i.e. the integration is well defined.
\begin{proof}
For an Euclidean vector bundle $\ga\to B$ we denote by $\widehat{\ga}$ the bundle that one obtains by fiberwise identifying the spheres in the associated unit disk bundle. One has an inclusion $B\to\widehat{\ga}$, $b\mapsto\{\text{sphere over }b\}$, and $\Th(\ga)$ can be written as $\widehat{\ga}/B$.

To prove Proposition \ref{gint} we can restrict ourselves to finite subcomplexes using cellular cohomology.
We have fiber bundles
\begin{equation}\label{nott}
\begin{split}
&(S^{m},*)\hto(\widehat{-\tau_m},E_m)\xto{q}E_m,\\
&(S^{k+m},*) \hto (\widehat{(B_m\x\R^{k+m})},B_m)\xto{r} B_m,\\
&(\Th(-\tau_m|_F),*) \hto \left(\bigcup_{b\in B_m}\Th(-\tau_m|_{p^{-1}(b)}),B_m\right) \xto{\tilde{r}} B_m,\\
&(\widehat{-\tau_m|_F},F)\hto(\widehat{-\tau_m},E_m)\xto{p\circ q}B_m.
\end{split}
\end{equation}
Observe that $\left(\bigcup_{b\in B_m}\Th(-\tau_m|_{p^{-1}(b)})\right)/B_m=\Th(-\tau_m)$. The umkehr map
\[
t_m\colon(\widehat{(B_m\x\R^{k+m})},B_m) \to \left(\bigcup_{b\in B_m}\Th(-\tau_m|_{p^{-1}(b)}),B_m\right)
\]
and the projection
\[
(\widehat{-\tau_m},E_m)\xto{proj} \left(\bigcup_{b\in B_m}\Th(-\tau_m|_{p^{-1}(b)}),B_m\right)
\]
are maps of fiber bundles over $B_m$. The fibers of the bundles 2 to 4 in \ref{nott} have the same cohomological dimension $k+m$. Hence the following diagram commutes:
\[
\xgrid=24mm
\ygrid=7mm
\cellpush=6pt
\Diagram
H^n(\widehat{-\tau_m},E_m)\aTo(2,-4)>{(p\circ q)_{\natural}} & \lTo^{proj^*}_{\cong} & H^n\left(\bigcup_{b\in B_m}\Th(-\tau_m|_{p^{-1}(b)}),B_m\right) & \rTo^{t_m^*} & H^n(\widehat{(B_m\x\R^{k+m})},B_m)\aTo(-2,-4) >{r_!}<{\cong}\\
\dTo^{q_!}_{\cong} &                      & \dTo^{\tilde{r}_!}(0,-4)    &  &   \\
H^{n-m}(E_m)                &                      &                                      &  &   \\
                            & \rdTo^{p_!} &                                      &  &   \\
                            &                      & H^{n-m-k}(B_m).                      &  &   \\
\endDiagram
\]
Here $q_!$ and $r_!$ are both given by inverse Thom isomorphisms because the respective fiber bundles are vector bundles. In addition $r_!$ can be identified with the suspension isomorphism for the respective bundle is trivial. Thus the composition
\begin{equation}\label{tr1}
r_! \circ t_m^* \circ (proj^*)^{-1} \circ q_!^{-1}
\end{equation}
can be identified with \ref{tink}.
\end{proof}

\section{The cohomology of the fiber bundle}\label{fiberbundle}

Unless otherwise stated we consider homology and cohomology with $\Z_2$-coefficients. The lower index of homology and cohomology classes denote their degree.

\begin{thm}\label{a} Consider the fiber bundle
\[
\C\text{\emph{P}}^2\hto B(S(U(2)U(1))\rtimes\Z_2)\xto{\pi} B(SU(3)\rtimes \Z_2).
\]
For an appropriate choice of cohomology classes $x_i$ and $y_i$ we have
\begin{enumerate}
\item $H^*(B(S(U(2)U(1))\rtimes\Z_2))\cong\Z_2[x_1,x_2,x_4]$.
\item $H^*(B(SU(3)\rtimes\Z_2))\cong \Z_2[y_1,y_4,y_6]$.
\item $\pi^*\colon H^*(B(SU(3)\rtimes \Z_2)) \to H^*(B(S(U(2)U(1))\rtimes\Z_2))$ is given by
$y_1\mapsto x_1,\ y_4\mapsto x_2^2+x_4,\ y_6\mapsto x_2x_4$.
\item $Sq^1(y_1)=y_1^2,\ Sq^1(y_4)=0,\ Sq^1(y_6)=y_1y_6$.
\item The total Stiefel-Whitney class of the tangent bundle along the fiber is $\gw(\tau)=1+(x_1^2+x_2)+x_1x_2+x_4$.
\item $\pi_!\colon H^n(BH)\to H^{n-4}(BG)$ is modulo $y_6$ given by
\[
\pi_!(x_1^ax_2^bx_4^c)\equiv\begin{cases}y_1^ay_4^{d-1},\ &\text{if}\ b=2d>0,\, c=0\\
                                                   y_1^ay_4^{c-1},\ &\text{if}\ b=0,\, c>0\\
                                                                0,\ &\text{otherwise}.
\end{cases}
\]
\end{enumerate}
\end{thm}

We chose as generators of the cohomology of $BU(i)$ and $BSU(i)$ the Chern classes of the corresponding universal bundles.

We begin with the first three statements of Theorem \ref{a}. Apply the Leray-Serre spectral sequence to the following fiber bundles and the map $\pi$:
\begin{equation}\label{faserung}
\xgrid=22mm
\ygrid=7mm
\cellpush=6pt
\Diagram
BS(U(2)U(1))                &  \rTo^{\tilde{\pi}}  &  BSU(3)                \\
\dTo                        &                      &  \dTo                  \\
B(S(U(2)U(1))\rtimes \Z_2)  &  \rTo^{\pi}          &  B(SU(3)\rtimes \Z_2)  \\
\dTo                        &                      &  \dTo                  \\
B\Z_2                       &  \rTo^{\id}          &  B\Z_2.                \\
\endDiagram
\end{equation}
The spectral sequences collapses in both cases and one deduces
\begin{align*}
H^*(B(S(U(2)U(1))\rtimes\Z_2))&\cong H^*(B\Z_2)\ot H^*(BS(U(2)U(1))) \\
                              &\cong \Z_2[x_1]\ot\Z_2[x_2,x_4]\quad \text{and}\\
H^*(B(SU(3)\rtimes\Z_2))&\cong H^*(B\Z_2)\ot H^*(BSU(3)) \\
                        &\cong \Z_2[y_1]\ot\Z_2[y_4,y_6].
\end{align*}
Here we used the homotopy equivalence $BS(U(2)U(1))\xto{incl}B(U(2)U(1))\xto{proj}BU(2)$ to get generators of $H^*(BS(U(2)U(1))$. Let us describe $\tilde{\pi}$. Writing $H^*(B(U(2)U(1)))\cong\Z_2[a_2,a_4,b_2]$, then the inclusion $BS(U(2)U(1))\hto B(U(2)U(1))$ induces, by construction, $a_2\mapsto x_2$ and $a_4\mapsto x_4$. Besides, $b_2\mapsto x_2$ as one sees by taking an inclusion $BU(1)\hto BS(U(2)(U(1)))$ into the second factor. Consider now the inclusions
\[
\xgrid=17mm
\ygrid=7mm
\cellpush=6pt
\Diagram
BS(U(2)U(1))        &  \rTo   &  B(U(2)U(1))  \\
\dTo^{\tilde{\pi}}  &         &  \dTo^j  \\
BSU(3)              &  \rTo   &  BU(3)       \\
\endDiagram
\]
and note that we understand $j$ via the product formula for Chern classes. We conclude $\tilde{\pi}^*(y_4)= x_2^2+x_4$ resp. $\tilde{\pi}(y_6)= x_2x_4$ and understand $\pi$ by exploiting the naturality of the spectral sequence.


The fourth point og Theorem \ref{a} follows from
\begin{lem}\label{l32} Let $H^*(B(U(3)\rtimes \Z_2))\cong\Z_2[c_1,c_2,c_4,c_6]$. Then we have
\[
Sq^1(c_4)=0\quad \text{and}\quad Sq^1(c_6)=c_1c_6.
\]
\end{lem}
\begin{proof}
Write $H^*(B\Z_2^4)\cong\Z_2[\ga,\gb,\gc,\gd]$. We put
\[
A:=\ga^2+\ga\gd,\ B:=\gb^2+\gb\gd,\ C:=\gc^2+\gc\gd
\]
and observe
\begin{equation}\label{arg9}
Sq^1(A)=A\gd,\ Sq^1(B)=B\gd\ \text{and}\ Sq^1(C)=C\gd.
\end{equation}
Consider now $(U(3)\rtimes\Z_2)\subset O(6)$. The inclusion
\begin{align*}
&i\colon\Z_2^4\hto(U(3)\rtimes\Z_2),\\
&(-1,1,1,1)\mapsto\left(\begin{smallmatrix}-1&0&0&0&0&0\\0&-1&0&0&0&0\\0&0&1&0&0&0\\
                                           0&0&0&1&0&0\\0&0&0&0&1&0\\0&0&0&0&0&1\end{smallmatrix}\right),\
(1,-1,1,1)\mapsto\left(\begin{smallmatrix}1&0&0&0&0&0\\0&1&0&0&0&0\\0&0&-1&0&0&0\\
                                           0&0&0&-1&0&0\\0&0&0&0&1&0\\0&0&0&0&0&1\end{smallmatrix}\right),\\
&(1,1,-1,1)\mapsto\left(\begin{smallmatrix}1&0&0&0&0&0\\0&1&0&0&0&0\\0&0&1&0&0&0\\
                                           0&0&0&1&0&0\\0&0&0&0&-1&0\\0&0&0&0&0&-1\end{smallmatrix}\right),\
(1,1,1,-1)\mapsto\left(\begin{smallmatrix}1&0&0&0&0&0\\0&-1&0&0&0&0\\0&0&1&0&0&0\\
                                           0&0&0&-1&0&0\\0&0&0&0&1&0\\0&0&0&0&0&-1\end{smallmatrix}\right)
\end{align*}
induces the map $(Bi)^*\colon H^*(B(U(3)\rtimes\Z_2))\to H^*(B\Z_2^4)$,
\[
c_1\mapsto \gd,\ c_2\mapsto A+B+C,\ c_4\mapsto AB+BC+AC,\ c_6\mapsto ABC.
\]
This follows from the observation that the inclusion
\[
\Z_2^2\to U(1)\rtimes\Z_2\cong O(2),\quad
(-1,1)\mapsto \left(\begin{smallmatrix}-1&0\\0&-1\end{smallmatrix}\right),\quad
(1,-1)\mapsto \left(\begin{smallmatrix}1&0\\ 0&-1\end{smallmatrix}\right)
\]
with $H^*(B\Z_2^2)\cong\Z_2[u]\otimes\Z_2[v]$ and $H^*(BO(2))\cong\Z_2[w_1,w_2]$ induces the map
\[
w_1\mapsto v\quad \text{and}\quad w_2\mapsto u^2+uv.
\]
Now with \ref{arg9} we obtain
\begin{align*}
(Bi)^*(Sq^1(c_4))&=Sq^1((Bi)^*c_4)              & (Bi)^*(Sq^1(c_6))&=Sq^1((Bi)^*c_6)\\
                 &=Sq^1(AB+BC+AC)               &                  &=Sq^1(ABC)\\
                 &=0,                           &                  &=(Bi)^*(c_1c_6).
\end{align*}
Since $(Bi^*)$ is injective the lemma holds.
\end{proof}


To study the Stiefel-Whitney class of the tangent bundle along the fiber $\tau$ let $\mathfrak{h}^{\bot}\subset\mathfrak{g}$ be an $\ad H$-invariant subspace complementary to $\mathfrak{h}$. It is not difficult to see that $\tau$ is isomorphic to
\[
EG\x_H\mathfrak{h}^{\bot} \to EG/H=BH.
\]

First we show
\begin{lem}\label{analog} Let $B(U(2)U(1)\rtimes\Z_2)\cong\Z_2[d_1,a_2,b_2,a_4]$ and $\tilde{\tau}$ be the tangent bundle along the fiber of
\[
\C\text{\emph{P}}^2\hto B(U(2)U(1)\rtimes\Z_2) \to B(U(3)\rtimes\Z_2).
\]
Then we have:
\[
w(\tilde{\tau})=1+(d_1^2+a_2)+d_1a_2+(a_4+b_2^2+a_2b_2).
\]
\end{lem}
\begin{proof}
The Lie algebra $\mathfrak{u(3)}$ of $U(3)\rtimes\Z_2$ can be identified with the complex, skew-Hermitian 3x3-matrices. Let $\mathfrak{u(2)u(1)}^{\bot}\cong\C^2$ be
\[
\begin{pmatrix} 0 & 0 & v_1 \\ 0 & 0 & v_2 \\ \overline{v_1} & \overline{v_2} & 0 \end{pmatrix},\
v=\begin{pmatrix} v_1 \\  v_2 \end{pmatrix}\in\C^2,
\]
which is an $\ad U(2)U(1)$-invariant subspace of $\mathfrak{u(3)}$ complementary to $\mathfrak{u(2)u(1)}$.
Now $\tilde{\tau}$ is given by
\[
E(U(2)U(1)\rtimes\Z_2)\x_{U(2)U(1)\rtimes\Z_2}\mathfrak{u(2)u(1)}^{\bot}
\]
with
\[
(P,z)\in U(2)U(1)\mapsto(\C^2\ni v\mapsto Pvz^{-1})
\]
and $\rtimes\Z_2$ as complex conjugation.

Let further $B\Z_2^4\cong\Z_2[\ga,\gb,\gc,\gd]$ and $A=\ga^2+\ga\gd, B=\gb^2+\gb\gd, C=\gc^2+\gc\gd$.
Similarly to the proof of statement 4 we consider
\[
(U(2)U(1)\rtimes\Z_2)\subset O(6)\quad\text{and}\quad i\colon\Z_2^4\hto((U(2)U(1))\rtimes\Z_2).
\]
Then $i$ induces $(Bi)^*\colon H^*(B(U(2)U(1)\rtimes\Z_2))\to H^*(B\Z_2^4)$,
\begin{equation}\label{lem1}
d_1\mapsto \gd,\ a_2\mapsto A+B,\ a_4\mapsto AB,\ b_2\mapsto C.
\end{equation}
The representation of $U(2)U(1)\rtimes\Z_2$ restricted to $\Z_2^4$ is given by
\[
\mathfrak{u(2)u(1)}^{\bot}|_{\Z_2^4}\cong(R_1 \ot R_3^{-1}) \oplus (R_1\ot R_3^{-1}\ot R_4) \oplus (R_2\ot R_3^{-1}) \oplus (R_2\ot R_3^{-1}\ot R_4),
\]
where $R_i$ denotes the one-dimensional real representation
\[
\Z_2^4\to Aut(\R), (\rho_1,\rho_2,\rho_3,\rho_4)\mapsto (x\mapsto \rho_ix).
\]
Hence
\begin{align*}
w(i^*(\tilde{\tau}))&=(1+\ga+\gc)(1+\ga+\gc+\gd)(1+\gb+\gc)(1+\gb+\gc+\gd)\\
                    &=1+(A+B+\gd^2)+((A+B)\gd)+(AB+C^2+(A+B)C).
\end{align*}
Combining this with \ref{lem1} yields Lemma \ref{analog}.
\end{proof}

The claim of point 5 follows since the bundle map
\[
incl\colon B(S(U(2)U(1))\rtimes\Z_2)\to B(U(2)U(1)\rtimes\Z_2)
\]
induces a map $\tau\to\tilde{\tau}$. We have $incl^*(d_1)=x_1$ and finally
\begin{align*}
w(\tau)&=incl^*(w(\tilde{\tau}))\\
       &=1+(x_1^2+x_2)+x_1x_2+x_4.
\end{align*}


We address the last statement of Theorem \ref{a}. Since the Leray-Serre spectral sequence associated to $\CP^2\hto BH\to BG$ collapses $H^*(BH)$ becomes via $\pi$ a free $H^*(BG)$-module with basis $\{1,x_2,x_2^2\}$ (cf. \cite{maccl}, proof of Theorem 5.10). Thus any element $x\in H^*(BH)$ has an unique description $a(x)1+b(x)x_2+c(x)x_2^2$ with $a(x),\, b(x),\, c(x) \in H^*(BG)$.

The integration along the fiber $\pi_!\colon H^n(BH)\to H^{n-4}(BG)$ is given on the basis elements as follows: For dimensional reasons we have $\pi_!(1)=\pi_!(x_2)=0$. Since in the degree of the dimension of the fiber $\pi_!$ can be identified with the inclusion $H^4(BH)\to H^4(\CP^2)$, i.e. $\pi_!$ is surjective, we conclude that $\pi_!(x_2^2)=1\in H^0(BG)$. Then we compute
\begin{align*}
x_2^{2n}  &=(x_2^2+x_4)^{n-1}\cdot x_2^2+ (x_2x_4)  \cdot \textstyle{\sum_{s,t}}(x_2^s+x_4^t)\quad (\text{for}\ s,t\geq 0)\\
          &=\pi^*(y_4^{n-1}) \cdot x_2^2+ \pi^*(y_6)\cdot \textstyle{\sum_{s,t}}(x_2^s+x_4^t),\\
x_2^{2n+1}&=(x_2^2+x_4)^{n}  \cdot x_2+   (x_2x_4)  \cdot \textstyle{\sum_{s',t'}}(x_2^{s'}+x_4^{t'})\quad (\text{for}\ s',t'\geq 0)\\
          &=\pi^*(y_4^{n})   \cdot x_2+   \pi^*(y_6)\cdot \textstyle{\sum_{s',t'}}(x_2^{s'}+x_4^{t'}),\\
x_4^n     &=(x_2^2+x_4)^{n}  \cdot 1+     (x_2x_4)  \cdot \textstyle{\sum_{s'',t''}}(x_2^{s''}+x_4^{t''}) + x_2^{2n}\quad (\text{for}\ s'',t''\geq 0)\\
          &=\pi^*(y_4^{n})   \cdot 1+     \pi^*(y_6)\cdot \textstyle{\sum_{s'',t''}}(x_2^{s''}+x_4^{t''})\\
          &+\pi^*(y_4^{n-1}) \cdot x_2^2+ \pi^*(y_6)\cdot \textstyle{\sum_{s,t}}(x_2^s+x_4^t).
\end{align*}
Since $\pi_!$ is a $H^*(BG)$-module map the claim follows.

\section{MSO-module spectra}\label{smso}

The proofs of Theorem \ref{wide} and \ref{adam} require an understanding of $H_*(MSO\wedge\Sigma^4BG_+)$ and $H_*(MSO)$ as comodules over $A_*$. An important observation will be that $MSO\wedge\Sigma^4BG_+$ and $MSO$ are $MSO$-module spectra.

We shall consider $\Z_2$-vector spaces which carry additional structures as modules resp. comodules over algebras resp. coalgebras. The general algebraic context is explained in \cite{milmo}. Provided that no ring is mentioned we understand tensor products always over $\Z_2$. As usual let $A^*=H^*(H\Z_2)$ denote the Steenrod algebra of all stable cohomological operations. $A^*$ becomes via the spectrum multiplication $H\Z_2\wedge H\Z_2\to H\Z_2$ a coalgebra and, as one can show, a Hopf algebra. Its dual $A_*$ is a polynomial algebra
\[
A_*\cong\Z_2[\zeta_1,\zeta_2,\dots],\quad deg(\zeta_i)=2^i-1
\]
with comultiplication
\[
\zeta_a\mapsto\sum_{b+c=a}\zeta_b\ot\zeta_c^{2^b}.
\]
Here the $\zeta_i$ denote the Hopf algebra conjugates of Milnor's generator $\xi_i$ (cf. \cite{switzer}, Theorem 18.20).

For the following description of $H_*(MSO)$ and $H^*(MSO)$ we define the subalgebra and $A_*$-comodule $L_*:=\Z_2[\zeta_1^2,\zeta_2,\zeta_3,\dots]\subset A_*$ and the quotient Hopf algebra $A(0)_*:=A_*/(\zeta_1^2,\zeta_2,\zeta_3,\dots)$. Then one has: $A_*\boxempty_{A(0)_*}\Z_2= L_*$, where $\boxempty_{A(0)_*}$ denotes the cotensor product of $A(0)_*$-comodules. Dually we have the quotient coalgebra and $A^*$-module $L^*:=A^*/A^*Sq^1$ and the sub-Hopf algebra $A(0)^*:=\bigwedge_{\Z_2}(Sq^1)$. Then one has: $A^*\ot_{A(0)^*}\Z_2= L^*$ (cf. \cite{switzer}, Proposition 20.15).

We remind the reader that the ring spectrum structure of $MSO$ induces the Pontrjagin product resp. -coproduct
\begin{align*}
&H_*(MSO)\ot H_*(MSO)\to H_*(MSO)\quad \text{resp.}\\
&H^*(MSO)\to H^*(MSO)\ot H^*(MSO).
\end{align*}

Although - since working with $\Z_2$-coefficients - statements about the $A_*$-comodule structure resp. algebra structure of the homology dualize to respective ones about the $A^*$-module structure and coalgebra structure of the cohomology, we occasionally switch between the homology and cohomology perspective. We cite the following result \cite{pengelley} about the homology of MSO:
\begin{thm}\label{mso}
For an appropriate choice of elements $u_n\in H_n(MO)$ one has:
\begin{enumerate}
	\item $H_*(MO)\cong\Z_2[u_1,u_2,\dots]$.
	\item Define for $n\geq 2$
	      \[
	      v_n:=\begin{cases}u_{n/2}^2,\      &\text{if}\ n=2^i\\
	                        u_n+u_1u_{n-1},\ &\text{if}\ n=2k,\ k\neq2^i\\
	                        u_n,\            &\text{if}\ n=2k-1,\end{cases}
	      \]
	      then $H_*(MSO)\cong\Z_2[v_2,v_3,\dots]\subset H_*(MO)$.
	\item Let $C\subset H_*(MSO)$ be defined as $\Z_2[v_4,v_5,v_6,v_8,\dots,v_n,\dots]$, $n\geq4$ and $n\neq 2^i-1$. Then $C$ is a
subalgebra resp. $A_*$-subcomodule and $H_*(MSO)$ is as algebra resp. comodule over $A_*$ isomorphic to $L_*\ot C$.
\end{enumerate}
\end{thm}
We note that the tensor product of algebras becomes an algebra by componentwise multiplication and the tensor product of $A_*$-comodules becomes a $A_*$-comodule via componentwise comultiplication together with the multiplication of $A_*$.

From the theorem it follows that $H_*(MSO)$ is an extended module over $L_*$. The $L_*$-module- and $A_*$-comodule structures are compatible in the sense that the multiplication $L_*\ot H_*(MSO)\to H_*(MSO)$ is a map of $A_*$-comodules. This holds because $L_*\ot L_*\to L_*$ is a $A_*$-comodule map since $L_*\subset A_*$ and $A_*$ is a Hopf algebra. Now we use the structure as $L_*$-module to give a description of $H_*(MSO)$ as a comodule over $A_*$.

$L_*$ is an augmented algebra and we define the $L_*$-indecomposable quotient of a $L_*$-module $M$ by $\overline{M}:=M/\mu(I(L_*)\ot M)= \Z_2\ot_{L_*} M$, where $I(L_*)$ denotes the augmentation ideal and $\mu\colon L_*\ot M\to M$ the module multiplication.

Since $A(0)_*=A_*/I(L_*)$ we see that $I(L_*)$ and hence $\overline{H_*(MSO)}$ become a $A(0)_*$-comodule. By means of Proposition 5.4 from \cite{stolz} one gets
\begin{cor}
$H_*(MSO)$ is as $A_*$-comodule and $L_*$-module isomorphic to $A_*\boxempty_{A(0)_*}\overline{H_*(MSO)}$. Here $A_*\boxempty_{A(0)_*}\overline{H_*(MSO)}\subset A_*\ot\overline{H_*(MSO)}$ carries the structure of an extended $A_*$-comodule and $L_*$-module.
\end{cor}

This result can be generalized to $MSO$-module spectra. Let $X$ be a $MSO$-module spectrum with multiplication $\mu\colon MSO\wedge X\to X$. The homology of $X$ becomes a $L_*$-module by
\[
\mu_*\colon H_*(MSO)\ot H_*(X)\to H_*(X)
\]
restricted to $L_*\ot H_*(X)$, $L\subset H_*(MSO)$. Again with Proposition 5.4. from \cite{stolz} we obtain the $A_*$-comodule and $L_*$-module isomorphism
\begin{equation}\label{alg}
H_*(X)\cong A_*\boxempty_{A(0)_*}\overline{H_*(X)}.
\end{equation}
This construction is functoriel: If $X$ and $Y$ are $MSO$-module spectra a $MSO$-module spectrum map $f$ can be written as
\begin{equation}\label{alg2}
H_*(X)\cong A_*\boxempty_{A(0)_*}\overline{H_*(X)}\xto{\id\boxempty\overline{f}}A_*\boxempty_{A(0)_*}\overline{H_*(Y)}\cong H_*(Y).
\end{equation}

\section{The split surjection in homology}\label{splitsurjection}

In this section we prove Theorem \ref{wide}:
\[
\widehat{T}_*\colon H_*(MSO\wedge\Sigma^4BG_+) \to H_*(\widehat{MSO})
\]
is a split surjection of $A_*$-comodules.

Recall the diagram
\[
\xgrid=16mm
\ygrid=7mm
\cellpush=6pt
\Diagram
                       &                     &  \widehat{MSO}&         &     \\
                       & \ruTo^{\widehat{T}} &  \dTo_i       &         &     \\
MSO\wedge\Sigma^4BG_+  & \rTo^T              &  MSO          & \rTo^U  & H\Z.\\
\endDiagram
\]
Since $H_*(H\Z)\cong L_*$ (cf. \cite{maccl}, Theorem 6.19) and $U_*$ is a $A_*$-comodule map which is nontrivial on $H_0$ we conclude that $U_*|_{L_*}$ is a isomorphism.

It follows from the surjectivity of $U_*$ and by considering the long exact sequence
\[
\dots \xto{i_*} H_n(MSO) \xto{U_*} H_n(H\Z) \xto{\partial} H_{n-1}(\widehat{MSO}) \xto{i_*} H_{n-1}(MSO) \xto{U_*} \dots
\]
associated to the cofibration $\widehat{MSO} \to MSO \to H\Z$ that $i_*$ is injective. Hence we can identify $H_*(\widehat{MSO})$ with $\ker U_*$ and have to show that $T_*$ is split surjective onto $\ker U_*$.

\begin{lem}\label{tgegeben} $T$ is given by
\[
MSO\wedge\Sigma^4BG_+ \xto{\id\wedge t} MSO \wedge M(-\tau) \xto{\id\wedge Mv} MSO\wedge MSO \xto{\mu} MSO,
\]
where $t$ denotes the umkehr map \ref{kleintt}, $Mv$ the Thom class of $-\tau$ and $\mu$ the multiplication of $MSO$.
\end{lem}
\begin{proof} The transfer $\Omega_{n-4}(BG)\to\Omega_n(BH)$ is induced (cf. Proposition \ref{hint}) by
\begin{align*}
        MSO\wedge\Sigma^4BG_+&\xto{\id\wedge t} MSO\wedge M(-\tau)\\
                             &\xto{\id\wedge M\Delta} MSO\wedge M(-\tau)\wedge BH_+\\
                             &\xto{\id\wedge Mv\wedge\id} MSO\wedge MSO\wedge BH_+\\
                             &\xto{\mu\wedge\id} MSO\wedge BH_+.\\
\end{align*}
Note that $\tau$ resp. $-\tau$ is oriented with respect to $MSO$ because $G$ acts orientation preserving on $\CP^2$ (in fact, each diffeomorphism of $\CP^2$ is orientation preserving as its first Pontrjagin class is non zero). One obtains $T$ after taking the projection $MSO\wedge BH_+\to MSO$. \end{proof}

$T$ is a map of $MSO$-modules. Moreover, since $U$ is a ring spectrum map $H\Z$ becomes via
$MSO\wedge H\Z \xto{U\wedge\id} H\Z\wedge H\Z\xto{mult} H\Z$ a $MSO$-module and also $U$ a map of $MSO$-modules.

By means of \ref{alg2} the induced maps of $T$ and $U$ in homology can be written as
\[
A_*\boxempty_{A(0)_*}\overline{H_*(MSO\wedge\Sigma^4BG_+)} \xto{\id\boxempty \overline{T_*}} A_*\boxempty_{A(0)_*}\overline{H_*(MSO)} \xto{\id\boxempty \overline{U_*}} A_*\boxempty_{A(0)_*}\overline{H_*(H\Z)}.
\]
Now
\[
\overline{U_*}\colon\overline{H_*(MSO)}\to\overline{H_*(H\Z)}=\Z_2
\]
can be identified with the augmentation of the algebra $\overline{H_*(MSO)}$. Bearing in mind that $\ker U_* =A_*\boxempty_{A(0)_*}(\ker\overline{U_*})$ Theorem \ref{wide} follows from

\begin{prop}\label{l} The $A(0)_*$-comodule map
\[
\overline{T_*}\colon \overline{H_*(MSO\wedge\Sigma^4BG_+)} \to \overline{H_*(MSO)}
\]
is split surjective onto the augmentation ideal $I(\overline{H_*(MSO)})$ of $\overline{H_*(MSO)}$.
\end{prop}

For the upcoming computation we switch to cohomology. Let $M$ be any $A(0)^*=\bigwedge_{\Z_2}(Sq^1)$ module. Since $Sq^1\circ Sq^1=0$ we can regard $Sq^1$ as a differential on $M$. The split exact sequence of $\Z_2$-vector spaces
\[
0\to \im Sq^1\to \ker Sq^1\to H(M,Sq^1)\to 0
\]
extends to a split exact sequence of $A(0)^*$-modules
\[
0\to A(0)^*\ot\im Sq^1\xto{i} M\xto{p} H(M,Sq^1)\to 0.
\]
Therefore, $M$ consists of a sum of free $A(0)^*$-summands $A(0)^*\ot\im Sq^1$ and $\Z_2$-summands $H(M,Sq^1)$. We note that an injection of $A(0)^*$-modules splits if it induces also an injection in $Sq^1$-homology.

\begin{proof}[Proof of Proposition \ref{l}]
Let $\overline{H^*(MSO)}=\Z_2\boxempty_{L^*}H^*(MSO)$ be the dual of $\overline{H_*(MSO)}$. We show: The dual of $\overline{T_*}$, the $A(0)^*$-module map
\begin{equation}\label{kkh}\begin{split}
\overline{H^*(MSO)}&\xto{\overline{\mu^*}}\overline{H^*(MSO)}\ot \overline{H^*(MSO)}\\
                   &\xto{\id\ot t^*(Mv)^*}\overline{H^*(MSO)}\ot H^*(\Sigma^4BG_+),
\end{split}\end{equation}
is restricted to the cokernel of the augmentation $\Z_2\to\overline{H^*(MSO)}$ split injective. In the coalgebra $(\overline{H^*(MSO)},\overline{\mu^*})$ we consider the primitive elements $P\overline{H^*(MSO)}$:
\[
x\in P\overline{H^*(MSO)} \Leftrightarrow x\in\overline{H^*(MSO)},\ \overline{\mu^*}(x)=1\ot x+x\ot 1.
\]
Then the split injectivity of \ref{kkh} follows from the injectivity of
\[
P\overline{H^*(MSO)}\xto{t^*(Mv)^*}H^*(\Sigma^4BG_+)
\]
and
\[
PH^*(\overline{H^*(MSO)},Sq^1)\xto{t^*(Mv)^*}H^*(H^*(\Sigma^4BG_+),Sq^1).
\]

We deduce from Theorem \ref{mso} that $P\overline{H^*(MSO)}=0$ for $n<4$ or $n=2^i-1$. Hence we show
\begin{prop}\label{prim} The map
\[
PH^n(MSO) \xto{(Mv)^*} H^n(M(-\tau)) \xto{t^*} H^{n-4}(BG)
\]
is injective for $n\geq 4$ and $n\neq 2^i-1$.
\end{prop}

Since $H^n(H^*(MSO),Sq^1)=0$ for $n\not\equiv 0\modu4$ (cf.\cite{switzer}, p. 513) we show
\begin{prop}\label{sqprim} The map
\[
PH^n(H^*(MSO),Sq^1) \xto{(Mv)^*} H^n(H^*(M(-\tau)),Sq^1) \xto{t^*} H^{n-4}(H^*(BG),Sq^1)
\]
is injective for $n\equiv0\modu4$.
\end{prop}

From these two propositions the split injectivity of \ref{kkh} and hence Proposition \ref{l} follow.
\end{proof}

Before proving Proposition \ref{prim} and \ref{sqprim} we make some remarks. The map $(Mv)^*$ is induced by the classifying map $v\colon BH\to BSO$ of $-\tau$. Now $BSO$ is a homotopy-commutative H-group with inverse $h$. If $\tilde{v}$ is the classifying map of $\tau$ we have $v=h\circ \tilde{v}$ and $Mv$ is given by
\[
M(-\tau)\xto{M\tilde{v}}MSO\xto{Mh}MSO.
\]
From the homotopy-commutativity of $BSO$ it follows that $Mh$ is a map of ring spectra. In addition, $Mh$ induces in cohomology an isomorphism, hence we can replace $Mv$ by $M\tilde{v}$.

Further one observes that the Thom isomorphism $H^*(BSO)\cong H^*(MSO)$ respects the coalgebra structure (cf. \cite{switzer}, Lemma 16.36). Therefore we identify
\[
H^*(M(-\tau))\xto{(M\tilde{v})^*}H^*(MSO)\quad\text{with}\quad H^*(BH)\xto{\tilde{v}^*}H^*(BSO).
\]
By Theorem \ref{a}, (5), we understand $\tilde{v}^*$.
Let us summarize:
\[
\Sigma^4BG_+\xto{t} M(-\tau) \xto{Mv} MSO
\]
in cohomology induces the map
\begin{align*}
\Z_2[\gw_2,\gw_3,\dots]& \xto{\tilde{v}^*} \Z_2[x_1,x_2,x_4] \to \Z_2[y_1,y_4,y_6]\\
                  \gw_2& \mapsto x_1^2+x_2 \\
                  \gw_3& \mapsto x_1x_2    \\
                  \gw_4& \mapsto x_4.
\end{align*}

The latter map is the composition of Thom isomorphism and $t^*$, following Proposition \ref{gint} it is just the integration along the fiber. Modulo $y_6$ we have (cf. Theorem \ref{a}, (6))
\[
\pi_!(x_1^ax_2^bx_4^c)\equiv\begin{cases}y_1^ay_4^{d-1},\ &\text{if}\ b=2d>0,\, c=0\\
                                       y_1^ay_4^{c-1}, \ &\text{if}\ b=0,\, c>0\\
                                       0,             \ &\text{otherwise}.
\end{cases}
\]
The map $\tilde{v}^*$ factors through $H^*(BSO)/(\gw_5,\gw_6,\dots)$ and it is enough to show the injectivity onto
$(H^*(BG))/(y_6)$. Therefore we shall consider the maps
\begin{equation}\label{einfach}
\Z_2[\gw_2,\gw_3,\gw_4]\xto{\tilde{v}^*}\Z_2[x_1,x_2,x_4]\xto{\pi_!}\Z_2[y_1,y_4]
\end{equation}
and use the labels $\tilde{v}^*$ and $\pi_!$ again.

From Theorem \ref{mso} it also follows that $PH^n(BSO)$ is isomorphic to $\Z_2$ for $n\geq 4$, $n\neq 2^i-1$ and zero otherwise. Likewise from \cite{switzer}, p. 513, it follows that $PH^n(H^*(BSO),Sq^1)$ is isomorphic to $\Z_2$ for $n\equiv 0\modu4$ and zero otherwise. Hence we have to find in $H^*(BSO)$ and $H^*(H^*(BSO),Sq^1)$ one primitive element in the respective degrees and show that it is not sent to zero under $\pi_!\circ\tilde{v}^*$. To proceed in this way we shall express the primitive elements by Stiefel-Whitney classes.


\begin{proof}[Proof of Proposition \ref{prim}] The coproduct of $H^*(BSO)$ is induced by the Whitney sum of the universal bundles. For $n=2^{i+1}$ the element $\gw_2^{2^i}$ is primitive because $\gw_2$ is primitive and
\begin{equation}\label{hochi}\begin{split}
\psi(\gw_2^{2^i})&=\psi(\gw_2)^{2^i}\\
                 &=(1\ot\gw_2+\gw_2\ot1)^{2^i}\\
                 &=(1\ot\gw_2^{2^i}+\gw_2^{2^i}\ot1)
\end{split}\end{equation}
(in the binomial formula modulo $2$ all other summands vanish).

To describe the other primitive elements one has to introduce certain polynomial $s_I$, $I=(i_1,\dots,i_r)$ some partition of $n$, in the elementary symmetric functions $\sigma_1,\sigma_2,\ldots$. (cf. for a discussion \cite{milsta}, \textsection 16). If we replace $\sigma_i$ by the Stiefel-Whitney classes of a vector bundle $\xi$ we obtain an element $s_I(\gw):=s_I(\gw_1(\xi),\dots,\gw_n(\xi))$. One can show that the coproduct $H^*(BO)\to H^*(BO)\ot H^*(BO)$ is given on $s_I(\gw)$ mod $2$ by (cf. \cite{milsta}, Lemma 16.2)
\begin{equation}\label{primu}
s_I(\gw)\mapsto \sum_{JK=I}s_J(\gw)\ot s_K(\gw),
\end{equation}
to be summed over all partitions $J$ and $K$ with juxtaposition $I$. This shows that $s_n(\gw)\in H^n(BO)$ is a primitive element. (In fact, $s_n$ is just defined to become this primitive element.) In $s_n(\sigma_1,\sigma_2,\ldots)$ always the summand $n\cdot\sigma_n$ occurs which means that for $n>1$ odd $s_n(\gw)\neq 0$ in $H^*(BSO)$. For $n\neq2^{i+1}$ even we can write $n=2^{j}l$ for some $l>1$ odd. Then again $s_{2^{j}l}(\gw)\equiv s_l(\gw)^{2^{j}}\modu2$ is the primitive element  $\neq0$.


We show that no primitive element is sent to zero.
This is easy for $\gw_2^{2^i}$:
\begin{align*}
\gw_2^{2^i}&\xmapsto{\tilde{v}^*}(x_1^2+x_2)^{2^i}\\
           &=\sum_{k=0}^{2^i}\binom{2^i}{k}x_1^{2k}x_2^{2^i-k}\\
           &=x_2^{2^i}+x_1^{2^{i+1}}\\
           &\xmapsto{\pi_!} y_4^{2^{i-1}-1}.
\end{align*}

For $n\neq2^{i+1}$ the proof is more extensive. We put
\[
B:=\Z_2[\gw_2,\gw_3,\gw_4] \xto{\tilde{v}^*} C:=\Z_2[x_1,x_2,x_4] \xto{\pi_!} D:=\Z_2[y_1,y_4,y_6]
\]
and consider descending filtrations
\[
F_kC:=\text{span}(x_1^ax_2^bx_4^c)+\text{ideal}(x_2x_4)\quad \text{and}\quad F_kD:=\text{span}(y_1^ay_4^b),\quad a\geq k.
\]
Then $\pi_!$ respects the filtration. We have
\begin{align*}
\tilde{v}^*(\gw_2^a\gw_3^b\gw_4^c)&=(x_1^2+x_2)^ax_1^bx_2^bx_4^c\\
                                  &=\begin{cases}x_1^bx_2^{a+b}+\ \text{elements of}\ F_{b+1}C &\ \text{if}\ c=0\\
                                                 x_1^{2a}x_4^c+\ \text{term}\cdot x_2x_4        &\ \text{if}\ b=0\\
                                                 \text{term}\cdot x_2x_4                        &\ \text{if}\ b,c\neq0.\end{cases}
\end{align*}
This suggests the following filtration for $B$:
\[
F_kB:=\text{span}(\gw_2^a\gw_3^b) + \text{span}(\gw_2^{a'}\gw_4^c)+ \text{ideal}(\gw_3\gw_4),\quad b\geq k,\ 2a'\geq k.
\]
Then also $\tilde{v}^*$ respects the filtration. We observe that
\begin{equation}\label{beo}
\pi_!(\tilde{v}^*(\gw_2^a\gw_3^b))\equiv y_1^by_4^{e-1}\modu{F_{b+1}D},\quad \text{if}\ a+b=2e.
\end{equation}

First let $n$ be odd and $\geq 4$, $\neq 2^i-1$. The polynomial $s_n(\gw)$ can be computed by means of Girard's formula (cf. \cite{milsta}, p. 195):
\begin{equation}\label{gir}
(-1)^ns_n(\gw)=\sum_{i_1+2i_2+\dots+ni_n=n}(-1)^{i_1+\ldots+i_n}\frac{n\cdot(i_1+\ldots+i_n-1)!}{i_1!\cdot\dots\cdot i_n!}\gw_1^{i_1}\dots \gw_n^{i_n}.
\end{equation}

We only consider summands with $i_m=0$ for $m>4$. Since $n$ is odd there always appears an odd $i_3$. Below in Lemma \ref{binom} it is proved that there exists natural numbers $a$ and $b$ with
\begin{equation}\label{bed}
a,b\ \text{odd},\ 2a+3b=n,\ \frac{n(a+b-1)!}{a!b!}\equiv 1\modu{2}.
\end{equation}
We pick the smallest $i_3=b$ and related $i_2=a$ such that condition \ref{bed} is satisfied. With this pair $a,b$ we have
\begin{equation}
s_n\equiv \sum_j\gw_2^{a_j}\gw_3^{b_j}+\gw_2^a\gw_3^b\modu{F_{b+1}}B,
\end{equation}
where $\sum_j\gw_2^{a_j}\gw_3^{b_j}$ consists of summands with $b_j<b$ and which satisfy \ref{bed} ($a_j$ as $a$ resp. $b_j$ as $b$) except the condition $a$ odd.

Since
\[
\tilde{v}^*(\gw_2^{a_j}\gw_3^{b_j})=\sum_{t=0}^{a_j}\binom{a_j}{t}x_1^{2t+b_j}x_2^{a_j-t+b_j}
\]
and $\binom{a_j}{t}\equiv 0\modu{2}$ for $t$ odd on the one hand and
$\pi_!(x_1^{2t+b_j}x_2^{a_j-t+b_j})=0$ for $a_j-t+b_j$ odd, which means $t$ even, on the other hand, it follows that
\[
\pi_!(\tilde{v}^*\left(\sum_j\gw_2^{a_j}\gw_3^{b_j}\right))=0
\]
and in this way (see \ref{beo})
\[
\pi_!(\tilde{v}^*(s_n))\equiv y_1^by_4^{e-1}\modu{F_{b+1}D},\quad a+b=2e.
\]
This shows the injectivity for $n$ odd.

Because of $s_{2n}\equiv s_n^2\modu2$ we obtain for $n=2n'$, $n'$ odd,
\[
s_{2n'}\equiv \sum_j\gw_2^{2a_j}\gw_3^{2b_j}+\gw_2^{2a}\gw_3^{2b}\modu{F_{2b+1}}B.
\]
Pick the smallest $b_j=:b'$ and related $a'$. Then we have $s_{2n'}\equiv\gw_2^{2a'}\gw_3^{2b'}\modu{F_{2b'+1}}B$ and so
\[
\pi_!(\tilde{v}^*(s_{2n'})\equiv y_1^{2b'}y_4^{a'+b'-1}\modu{F_{2b'+1}D}.
\]
In this manner we see the injection for all even $n$.
\end{proof}

\begin{lem}\label{binom}
Let $n\neq 2^k-1$ be odd. Then there exists natural numbers $a$ and $b$ which satisfy \ref{bed}.
\end{lem}
\begin{proof}
We chose an $i$ such that $2^i<n<2^{i+1}$ and after that a $1\leq j\leq i-1$ such that $2^{i+1}-2^{i-j+1}<n<2^{i+1}-2^{i-j}$. Here the condition $n\neq 2^k-1$ enters because only in this case a $j\leq i-1$ can be chosen such that $n<2^{i+1}-2^{i-j}$. Now put
\[
a:=3\cdot 2^i - 3\cdot 2^{i-j}-n,\ b:=n-(2^{i+1}-2^{i-j+1}).
\]
A short computation shows $a,b>0$, $a,b$ odd, $2a+3b=n$ and $a+b=2^{i-j}(2^j-1)$.
Next we consider
\[
\frac{n(a+b-1)!}{a!b!}=\frac{n}{a}\binom{a+b-1}{b}
\]
and observe
\begin{align*}
b&=n-(2^{i+1}-2^{i-j+1})\\
 &<2^{i+1}-2^{i-j}-(2^{i+1}-2^{i-j+1})\\
 &=2^{i-j}.
\end{align*}
Now the claim is implied by the following observation. For arbitrary $p,m>0$ we have
\begin{align*}
\binom{2^pm-1}{x}&=\frac{(2^pm-1)}{1}\cdot\frac{(2^pm-2)}{2}\cdot\ldots\cdot\frac{(2^pm-x)}{x}\\
                 &\equiv 1\modu{2}\quad\text{if}\ x<2^p.
\end{align*}
Now put $2^{i-j}(2^j-1)=2^pm$ and $b=x$.
\end{proof}

Finally we treat the $Sq^1$-homology. We note that the Thom isomorphism $\Phi$ for a vector bundle $p\colon\ga\to B$ with Thom class $u\in H^n(\Th(\ga))$ can be described by
\[
H^i(B)\ni x\mapsto (u\cup p^*x)\in \tilde{H}^{i+n}(\Th(\ga)).
\]
If $\ga$ is oriented in the usual sense $\Phi$ commutes with $Sq^1$ because $\Phi^{-1}(Sq^1(u))=w_1(\ga)=0$.

It follows that the identifications $H^*(MSO)\cong H^*(BSO)$ and $H^*(M(-\tau))\cong H^*(BH)$ induce isomorphisms in $Sq^1$-homology.

First we show
\begin{lem}\label{sq1bg} $H^*(H^*(BG),Sq^1)\cong \Z_2[y_4,y_6^2]$.\end{lem}
\begin{proof}
Following Theorem \ref{a} we have $H^*(BG)\cong\Z_2[y_1,y_4,y_6]$ and
\[
Sq^1(y_1)=y_1^2,\ Sq^1(y_4)=0,\ Sq^1(y_6)=y_1y_6.
\]
It follows that
\[
Sq^1(y_1^r)=\begin{cases}0,\\
y_1^{r+1},\end{cases}
Sq^1(y_4^r)=\begin{cases}0,\\
0,\end{cases}
Sq^1(y_6^r)=\begin{cases}0,&\text{\ if $r$ even}\\
            y_1y_6^r,&\text{\ if $r$ odd}.\end{cases}
\]
Hence $Sq^1(y_1^ry_4^sy_6^t)=0$ iff $r$ and $t$ are both even or both odd.
For $r,t$ both odd on the one hand and for $r>0,t$ both even on the other hand we likewise have
\[
Sq^1(y_1^{r-1}y_4^sy_6^t)=y_1^ry_4^sy_6^t,
\]
which means $y_1^ry_4^sy_6^t\equiv 0$ in $Sq^1$-homology. In addition one sees that $y_1^ry_4^sy_6^t\not\equiv 0$ for $r=0$, $t$ even and $s$ arbitrary.
\end{proof}

\begin{proof}[Proof of Proposition \ref{sqprim}]
Following \cite{switzer}, p. 513, we have
\begin{equation}\label{sq1bso}
H^*(H^*(BSO),Sq^1)\cong Z_2[w_2^2,w_4^2,\dots].
\end{equation}
Then \ref{einfach} can be written in $Sq^1$-homology as
\begin{equation}\label{abbsq1}
\Z_2[\gw_2^2,\gw_4^2] \xto{\tilde{v}^*} H^*(\Z_2[x_1,x_2,x_4],Sq^1) \xto{\pi_!} \Z_2[y_4].
\end{equation}

Let us determine the primitive elements in $H^*(H^*(BSO),Sq^1)$.
\begin{lem}
$s_{2t,2t}$ is the primitive element $\neq0$ in $H^{4t}(H^*(BSO),Sq^1)$.
\end{lem}
\begin{proof}
Because of $s_{2t,2t}\equiv s_{t,t}^2 \modu{2}$ it follows that $s_{2t,2t}$ is a $Sq^1$-cocycle. We have
\[
\psi\colon s_{2t,2t} \mapsto 1\ot s_{2t,2t} + s_{2t}\ot s_{2t} + s_{2t,2t}\ot 1.
\]
This means $s_{2t,2t}$ is a primitive element if $s_{2t}\ot s_{2t}$ is a $Sq^1$-coboundary.
We show that $Sq^1(s_{2t-1}\ot s_{2t})=s_{2t}\ot s_{2t}$ which in turn follows from $Sq^1(s_{2t})=0$ and $Sq^1(s_{2t-1})=s_{2t}$. The latter statement: We consider $H^*(BO(1))\cong\Z_2[\gw_1]$. The $(\gw_1^i)^*$ form a basis of the vector space $H_*(BO(1))$. If we denote the image of $(\gw_1^i)^*$ under the inclusion $H_*(BO(1))\hto H_*(BO)$ by $z_i$, we get (cf. \cite{switzer}, p. 384)
\[
H_*(BO)\cong\Z_2[z_1,z_2,\dots]
\]
as algebra. Dualizing this we see that
\[
PH^*(BO)\to H^*(BO)\to H^*(BO(1))
\]
is an isomorphism. Since the primitive element $Sq^1(s_{2n-1})$ is send to $Sq^1(\gw_1^{2t-1})=\gw_1^{2t}$ the claim follows.
\end{proof}

One has the formula (cf. \cite{waerden}, p. 85) $s_{p,p}=\frac{1}{2}(s_{2p}-s_ps_p)$. We obtain (with $\Z$-coefficients)
\begin{equation}\label{s2t}\begin{split}
s_{2t,2t}     &=\frac{1}{2}\left(s_{4t}-(s_{2t})^2\right)\\
              &\equiv\frac{1}{2}(\frac{4t(2t-1)!}{(2t)!}\sigma_2^{2t}\pm\frac{4t(t-1)!}{t!}\sigma_4^{t}\\
              &\quad   -(\pm\frac{2t(t-1)!}{t!}\sigma_2^{t}\pm{\it \frac{2t(t/2-1)!}{(t/2)!}\sigma_4^{t/2}})^2)
                                                \modu(\sigma_2\sigma_4,\sigma_1,\sigma_3,\sigma_5,\sigma_6,\dots)\\
              &\equiv\frac{1}{2}\left(2\sigma_2^{2t}\pm4\sigma_4^t-\left(\pm2\sigma_2^t\pm{\it 4\sigma_4^{t/2}}\right)^2\right)\\
              &\equiv\frac{1}{2}\left(-2\sigma_2^{2t}\pm4\sigma_4^t-{\it 16\sigma_4^t}\right)\\
              &\equiv                  -\sigma_2^{2t}\pm2\sigma_4^t-{\it 8\sigma_4^t},
\end{split}\end{equation}
where the italic summands only occur if $t$ is even. It follows that $s_{2t,2t}\equiv\gw_2^{2t}$ in $\Z_2[\gw_2^2,\gw_4^2]/(\gw_2^2\gw_4^2)$. Without considering the $Sq^1$-homology we have
\begin{equation}\label{gw24}\begin{split}
\pi_!(\tilde{v}^*(\gw_2^a\gw_4^c))&=\pi_!((x_1^2+x_2)^ax_4^c)\\
                                           &=\pi_!(x_2^ax_4^c+\ \text{term}\cdot x_1)\\
                                           &=\begin{cases}y_4^{a/2-1}+\ \text{term}\cdot y_1,\ &a\geq 2\ \text{even},\ c=0\\
                                                          y_4^{c-1}+\ \text{term}\cdot y_1,\ &a=0,\ c>0\\
                                                          \text{term}\cdot y_1,\ &\text{otherwise}.\end{cases}
\end{split}\end{equation}
Since term$\,\cdot\, y_1=0$ in $Sq^1$-homology it follows that $\pi_!\circ\tilde{v}\colon\Z_2[\gw_2^2,\gw_4^2]\to\Z_2[y_4]$ factors over $(\gw_2^2\gw_4^2)$ and finally that $\pi_!(\tilde{v}^*(s_{2t,2t}))=y_4^{t-1}$.
\end{proof}

\section{The Adams spectral sequence}\label{adams}

We show Theorem \ref{adam}: The Adams spectral sequence
\[
Ext_{A_*}^{s,t}(\Z_2,H_*(MSO\wedge\Sigma^4BG_+))\Longrightarrow\pi_{t-s}(MSO\wedge\Sigma^4BG_+)/T_{\nmid 2}
\]
collapses on the $E_2$-term.

By means of \ref{alg} we write
\begin{equation}\label{fmo}\begin{split}
H_*(MSO)\otimes H_*(\Sigma^4BG_+) &\cong A_*\boxempty_{A(0)_*}(\Z_2\ot_{L_*}(H_*(MSO)\otimes H_*(\Sigma^4BG_+))\\
                                  &\cong A_*\boxempty_{A(0)_*}((\Z_2\ot_{L_*} H_*(MSO))\otimes H_*(\Sigma^4BG_+))\\
                                  &\cong A_*\boxempty_{A(0)_*}(\overline{H_*(MSO)}\ot H_*(\Sigma^4BG_+)).\\
\end{split}\end{equation}
Since the structure of $H_*(MSO)\otimes H_*(\Sigma^4BG_+)$ as $L_*$-module is given by multiplication onto the first factor we can change the brackets as indicated.

If one wants to decide whether an Adams spectral sequence collapses or not one can ignore free  $A_*$-summands \cite{margolis}. Note that free $A_*$-summands of $H_*(MSO)\ot H_*(\Sigma^4BG_+)$ correspond to free $A(0)_*$-summands of $\overline{H_*(MSO)}\ot H_*(\Sigma^4BG_+)$.

It is convenient to switch briefly to the cohomological perspective. Above we saw that the non-free part of a $A(0)^*$-module is given by its $Sq^1$-homology. Hence we consider
\begin{align*}
        &H^*(\overline{H^*(MSO)}\ot H^*(\Sigma^4BG_+),Sq^1)\\
\cong{} &H^*(\overline{H^*(MSO)},Sq^1)\ot H^*(H^*(\Sigma^4BG_+),Sq^1).
\end{align*}
From \cite{switzer} p. 513, we take that the $Sq^1$-homology of $H^*(MSO)$ is concentrated in the 0 mod 4-degrees and that $H^*(A^*/A^*Sq^1,Sq^1)\cong\Z_2$. Since $H^*(MSO)\cong A^*/A^*Sq^1\ot\overline{H^*(MSO)}$ (cf. Theorem \ref{mso}, dual version (3)) it follows that $H^*(\overline{H^*(MSO)},Sq^1)$ is concentrated on the 0 mod 4-degrees.

Following Lemma \ref{sq1bg} $H^*(H^*(\Sigma^4BG_+),Sq^1)$ is also concentrated in the 0 mod 4-degrees.

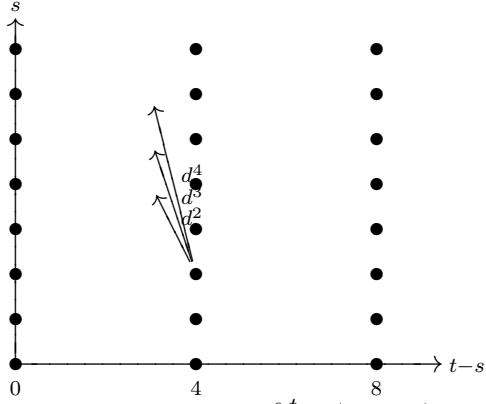
\begin{figure}[h]
\lpad=-5.5cm \tpad=4.5cm
$\Graph{0.6cm,1}{0.6cm,1}
\To (0,0) (10,0) \hd{\scriptstyle t-s}
\To (0,0) (0,8) \hd{\scriptstyle s}
\To \dt{-4pt}\dh{-4pt}^{d^2} (4,2) (3,4)
\To \dt{-4pt}\dh{-4pt}^{d^3} (4,2) (3,5)
\To \dt{-4pt}\dh{-4pt}^{d^4} (4,2) (3,6)
\Dot (0,0) \Dot (0,1) \Dot (0,2) \Dot (0,3) \Dot (0,4) \Dot (0,5) \Dot (0,6) \Dot (0,7) 
\Dot (4,0) \Dot (4,1) \Dot (4,2) \Dot (4,3) \Dot (4,4) \Dot (4,5) \Dot (4,6) \Dot (4,7) 
\Dot (8,0) \Dot (8,1) \Dot (8,2) \Dot (8,3) \Dot (8,4) \Dot (8,5) \Dot (8,6) \Dot (8,7) 
\Math{\scriptstyle 0} (0,-0.5) \Math{\scriptstyle 4} (4,-0.5) \Math{\scriptstyle 8} (8,-0.5)
\endGraph$
\\
\caption{$Ext_{A(0)_*}^{s,t}(\Z_2,V_*)$}\label{bass}
\end{figure}

\begin{proof}[Proof of Theorem \ref{adam}] Dualizing the above statements we get that after splitting off free $A_*$-summands $H_*(MSO)\ot H_*(\Sigma^4BG_+)$ is isomorphic to $A_*\boxempty_{A(0)_*} V_*$, where $V_*$ is a $A(0)_*$-comodule with elements only in the 0 mod 4-degrees. We apply a \emph{change-of-rings} isomorphism (cf. \cite{switzer}, Proposition 20.16) and obtain
\[
Ext_{A_*}^{s,t}(\Z_2,A_*\boxempty_{A(0)_*} V_*) \cong Ext_{A(0)_*}^{s,t}(\Z_2,V_*).
\]
We compute this term using \cite{switzer}, Proposition 19.8. There is a suitable resolution
\[
0\rightarrow\Z_2\rightarrow A(0)_*\ot\Z_2[q],\quad deg(q)=1,
\]
and one can show that $Ext_{A(0)_*}^{s,t}(\Z_2,\Z_2)\cong\Z_2[q]$, $q\in Ext_{A(0)^*}^{1,1}$ (cf. \cite{switzer}, p. 500-501). Hence
\[
Ext_{A(0)_*}^{s,t}(\Z_2,V_*)\cong\Z_2[q]\ot V_*,
\]
where the $(s,t)$-degree of $v\in V_*$ is $(0,deg(v))$. Since the differentials $d_n$ decrease the $t-s$-degree by $1$ it follows that $d_n=0$ for all $n$.
\end{proof}

To prove Theorem \ref{as}, i.e. the surjectivity of
\[
\widehat{T}_*\colon\pi_*(MSO \wedge\Sigma^4BG_+)\to\pi_*(\widehat{MSO}),
\]
we have to solve an extension problem. For that we need the statement from Sec. \ref{free} that $\widehat{T}_*$ becomes a surjection when it is composed with the projection onto $\Omega_*/\text{Tor}$.
\begin{proof}[Proof of Theorem \ref{as}]
It follows from Theorem \ref{wide} and \ref{adam} that $\widehat{T}$ induces a surjection
\[
E_{\infty}^{s,t}(MSO\wedge\Sigma^4BG_+)\to E_{\infty}^{s,t}(\widehat{MSO})
\]
on the $E_{\infty}$-terms of respective mod $2$ Adams spectral sequences. We consider the filtrations of the homotopy groups (cf. \cite{switzer}, Theorem 19.9):
\[
\xgrid=23mm
\ygrid=8mm
\cellpush=7pt
\Diagram
\pi_n(MSO\wedge\Sigma^4BG_+)/T_{\nmid 2}  &  \rTo^{\widehat{T}_*}&  \pi_n(\widehat{MSO})/T_{\nmid 2} \\
\dEq                                      &                      &  \dEq                             \\
A^{0,n}                                   &  \rTo                &  B^{0,n}                          \\
\bigcup                                   &                      &  \bigcup                          \\
A^{1,n+1}                                 &  \rTo                &  B^{1,n+1}                        \\
\bigcup                                   &                      &  \bigcup                          \\
A^{2,n+2}                                 &  \rTo                &  B^{2,n+2}                        \\
\bigcup                                   &                      &  \bigcup                          \\
\vdots                                    &                      &  \vdots                           \\
\endDiagram
\]
with
\[
A^{s,t}/A^{s+1,t+1}\cong E_{\infty}^{s,t}(MSO\wedge\Sigma^4BG_+)\quad \text{resp.}\quad B^{s,t}/B^{s+1,t+1}\cong E_{\infty}^{s,t}(\widehat{MSO})
\]
and
\[
\bigcap_{s=0}^{\infty}A^{s,n+s}=0\quad \text{resp.}\quad \bigcap_{s=0}^{\infty}B^{s,n+s}=0.
\]
Since $\pi_n(\widehat{MSO})\hto\pi_n(MSO)$ contains no odd torsion (cf. \cite{switzer}, Theorem 20.39) we have
\[
\pi_n(\widehat{MSO})/T_{\nmid 2}=\pi_n(\widehat{MSO}).
\]
We remind the reader that $\pi_n(\widehat{MSO})=\pi_n(MSO)$ for $n\geq 1$. In the proof of Proposition \ref{tfree} we construct $\CP^2$-bundles which are contained in $\im\widehat{T}_*$ and generate $\Omega_*/\text{Tor}$. In other words, the composition $\im\widehat{T}_*\hto\Omega_n\to\Omega_n/\text{Tor}_n$ is surjective. It follows that the composition
\[
\text{Tor}_n\hto\Omega_n\to\Omega_n/\im\widehat{T}_*
\]
is surjective, i.e. $\Omega_n/\im\widehat{T}_*= B^{0,n}/\im\widehat{T}_*$ is finite. Hence for the induced filtration
\[
\ldots\subset C_2\subset C_1\subset C_0 = B^{0,n}/\im\widehat{T}_*
\]
there exists an $s\geq 0$ such that $C_{s'}=0$ for all $s'>s$. Thus $A^{0,s}\to B^{s,0}\to C_s$ is surjective and step by step we conclude that
\[
\Omega_{n-4}(BG_+)/T_{\nmid 2}  \xto{\widehat{T}_*}  \Omega_n \xto{proj} \Omega_n/\im\widehat{T}_*
\]
is surjective. The claim follows.
\end{proof}

\section{The oriented cobordism ring modulo torsion}\label{free}

In this section we show that $\Omega_*/\text{Tor}$ is generated by $\CP^2$-bundles.
The idea of the proof originates from \cite{hpeh}, Sec. 4. Unless otherwise specified (co-)homology groups are now understood with $\Z$-coefficients.

For a vector bundle $\xi\to B$ now let $s_n(\xi)\in H^{4n}(B)$ denote the polynomial which arises if one replaces the elementary symmetric function by the Pontrjagin classes of $\xi$ (cf. \cite{milsta}, \textsection 16).

It is a crucial simplification compared to the torsion part that an explicit criterion is known which manifolds generate $\Omega_*/\text{Tor}$ (cf. \cite{stong}, p. 180).
\begin{thm}\label{erzeuger} $\Omega_*/\text{Tor}\cong\Z[M^4,M^8,\dots]$ and a manifold $M^{4n}$ serves as a generator if and only if
\[
S_n(M^{4n}):=\langle s_n(TM),[M^{4n}]\rangle=\begin{cases}\pm1,\ &\text{if}\ 2n+1\ \text{is not a prime power}\\
                                                          \pm p,\ &\text{if}\ 2n+1\ \text{is a power of some prime $p$.}
                                             \end{cases}
\]
\end{thm}

We shall construct manifolds which satisfy this condition. Fix a $n\geq1$. For $1\leq r\leq n$ we consider the complex 3-dimensional bundle $E_r:=\gc_1\x\gc_2\oplus\C$ over $Q:=\CP^{r-1}\x\CP^{2n-r-1}$, where $\gc_i$ denote the canonical complex line bundles and $\C$ the trivial bundle. The associated projective bundle $PE_r$ is a $\CP^2$-bundle with structure group $U(3)$ and a manifold of real dimension $4n$.

\begin{prop}\label{tfree} $PE_r$, $1\leq r\leq n$, generate $\Omega_{4n}/\text{Tor}_{4n}$. \end{prop}

First we compute the Pontrjagin numbers.
\begin{lem}\label{1l} $S_n(PE_r)=1+(-1)^{r+1}\binom{2n}{r}$. \end{lem}

\begin{proof} Let $G:=U(3)$ and $H:=U(2)U(1)$. Then $G$ acts via matrix multiplication on $\CP^2$ with isotropy group $H$. Hence $\CP^2\hto BH\to BG$ is the universal $\CP^2$-bundle with structure group $G$ and $PE_r$ can be written as the pullback
\[
\xgrid=15mm
\ygrid=7mm
\cellpush=6pt
\flexible
\Diagram
PE_r={} & f^*(BH)                 &  \rTo^F      &  BH          \\
        & \dTo^p                  &              &  \dTo^{\pi}  \\
        & Q                       &  \rTo^f      &  BG          \\
\endDiagram
\]
for suitable maps $f$ and $F$.
Let $\tau\to BH$ denote the tangent bundle along the fiber. The tangent bundle of $(PE_r)$ decomposes as $p^*(TQ)\oplus F^*(\tau)$ and one can compute
\begin{equation}\label{rech}\begin{split}
S_n(PE_r)&=\langle s_n(p^*(TQ)\oplus F^*(\tau)),[PE_r]\rangle\\
         &=\langle p^*(s_n(TQ))+F^*(s_n(\tau)),[PE_r]\rangle\\
         &=\langle p_{\natural}(p^*(s_n(TQ)))+p_{\natural}(F^*(s_n(\tau))),[Q]\rangle\\
         &=\langle f^*(\pi_{\natural}(s_n(\tau))),[Q]\rangle.
\end{split}\end{equation}
For the second line note the additivity of $s_n$ modulo elements of order 2 and $H^{4n}(PE_r)\cong\Z$. In the third line we applied $p_*(x\cap p^!(y))=p_!(x)\cap y$. The fourth line follows from $s_n(TQ)\in H^{4n}(Q)=0$ and the naturality of the integration.

Analogous to the proof of Lemma \ref{analog} we consider the $\ad H$-invariant subspace $\C^2\cong\mathfrak{h}^{\bot}\subset\mathfrak{g}$. Then $\tau$ can be written as $EG\x_H\mathfrak{h}^{\bot}$ with $H$-action
\[
(P,z)\in (U(2)U(1))\mapsto(\C^2\ni v\mapsto Pvz^{-1}).
\]
Now we determine $s_n(\tau)=s_n(EH\x_H\mathfrak{h}^{\bot})$. If one considers the inclusion $U(1)^3\hto U(2)U(1)\hto U(3)$ one can interpret $H^*(BH)$ and $H^*(BG)$ as subrings of $H^*(BU(1)^3)\cong\Z[x_1,x_2,x_3]$, $deg(x_i)=2$, as follows:
\begin{align*}
&H^*(BH)\subset\Z[x_1,x_2,x_3],\quad \text{polynomials symmetric in}\ x_1, x_2,\\
&H^*(BG)\subset\Z[x_1,x_2,x_3],\quad \text{polynomials symmetric in}\ x_1, x_2, x_3.
\end{align*}
Hence we write
\[
H^*(BH)=\Z[x_1,x_2,x_3]^{\Sigma_2}\quad \text{and}\quad H^*(BG)=\Z[x_1,x_2,x_3]^{\Sigma_3}.
\]
The representation of $H$ restricted to $U(1)^3$ is given by
\[
\mathfrak{h}^{\bot}|_{U(1)^3}\cong (C_1\ot C_3^{-1})\oplus (C_2\ot C_3^{-1}),
\]
where $C_i$ denotes the one dimensional complex representation
\[
U(1)^3\to Aut(\C), (z_1,z_2,z_3)\mapsto (x\mapsto z_ix).
\]
For a complex line bundle $\xi\to B$ one has $s_n(\xi)=p_1(\xi)^n$ (cf. e.g. \ref{gir}) and so $s_n(\xi)=c_1(\xi)^{2n}$.
In this way we obtain
\[
s_n(\tau)=(x_1-x_3)^{2n}+(x_2-x_3)^{2n}.
\]
Again the Leray-Serre spectral sequence at hand collapses and $H^*(BH)$ becomes a free $H^*(BG)$-module with basis $\{1,x_3,x_3^2\}$ and $\pi_!(a(x)1+b(x)x_3+c(x)x_3^2)=c(x)$. Now the following observation is helpful: The diagram
\begin{equation}\label{kd}
\xgrid=16mm
\ygrid=7mm
\cellpush=6pt
\flexible
\Diagram
\Z[x_1,x_2,x_3]^{\Sigma_2}  &  \rTo^{\pi_!} &  \Z[x_1,x_2,x_3]^{\Sigma_3} \\
\dTo^{\cdot x_1}            &                        &  \dTo^{\cdot\omega}         \\
\Z[x_1,x_2,x_3]             &  \rTo^A                &  \Z[x_1,x_2,x_3]            \\
\endDiagram
\end{equation}
commutes. Here $A$ is the map
\[
x_1^{\ga}x_2^{\gb}x_3^{\gc}\mapsto\sum_{\sigma\in\Sigma_3}\sign(\sigma) x_{\sigma(1)}^{\ga}x_{\sigma(2)}^{\gb}x_{\sigma(3)}^{\gc}
\]
and $\omega=A(x_1x_3^2)$. Since all maps are $\Z[x_1,x_2,x_3]^{\Sigma_3}$-module maps a computation on the basis shows the commutativity.

The classifying map $f\colon\CP^{r-1}\x\CP^{2n-r-1}\to BU(3)$ of $E_r$ resp. $PE_r$ is given by the inclusion
\[
\CP^{r-1}\x\CP^{2n-r-1}\hto\CP^{\infty}\x\CP^{\infty}\cong BU(1)^2\xto{i}BU(3),
\]
where $i$ is induced by $U(1)^2\to U(3)$, $(\ga,\gb)\mapsto \left(\begin{smallmatrix} \ga&0&0\\
                                                                                       0&\gb&0\\
                                                                                       0&0&1\end{smallmatrix}\right)$.
Then we have
\begin{align*}
(Bi)^*(A(x_1s_n(\tau)))
& =(Bi)^*\left(\sum_{\sigma\in\Sigma_3}\sign(\sigma)x_{\sigma(1)}(x_{\sigma(1)}-x_{\sigma(3)})^{2n}\right)\\
& =(Bi)^*(x_1(x_1-x_3)^{2n}-x_1(x_1-x_2)^{2n}-x_3(x_3-x_1)^{2n}\\
& -x_2(x_2-x_3)^{2n}+x_3(x_3-x_2)^{2n}+x_2(x_2-x_1)^{2n})\\
& =x_1^{2n+1}-x_2^{2n+1}-(x_1-x_2)^{2n+1}\quad(\text{da $f^*(x_3)=0$})\\
& =(x_1-x_2)\left(\sum_{k=0}^{2n}x_1^kx_2^{2n-k}-\sum_{k=0}^{2n}(-1)^k\binom{2n}{k}x_1^kx_2^{2n-k}\right)\\
& =(x_1-x_2)\left(\sum_{k=1}^{2n-1}x_1^kx_2^{2n-k}-\sum_{k=1}^{2n-1}(-1)^k\binom{2n}{k}x_1^kx_2^{2n-k}\right)\\
& =(x_1-x_2)x_1x_2\left(\sum_{k=1}^{2n-1}\left(1-(-1)^k\binom{2n}{k}\right)x_1^{k-1}x_2^{2n-k-1}\right).\\
\end{align*}
For the first equality observe that $A(x_1(x_2-x_3)^{2n})=0$.

Because of $(Bi)^*(\omega)=(x_1-x_2)x_1x_2$ one get with \ref{kd}
\[
(Bi)^*(\pi_!(s_n(\tau)))=\sum_{k=1}^{2n-1}\left(1-(-1)^k\binom{2n}{k}\right)x_1^{k-1}x_2^{2n-k-1}.
\]
Since
\[
\langle x_1^{k-1}x_2^{2n-k-1},[\CP^{r-1}\x\CP^{2n-r-1}]\rangle=\begin{cases}1,\ &\text{if}\ k=r\\
                                                                            0,\ &\text{otherwise},\end{cases}
\]
the claims follows now by means of \ref{rech}.
\end{proof}


For a finite sequence $a,b,\ldots\in\Z$ the Euclidean algorithm yields coefficients $A,B,\ldots\in\Z$ such that $Aa+Bb+\dots=\text{gcd}(a,b,\dots)$. Since $S_n(M\sqcup N)=S_n(M)+S_n(N)$ Proposition \ref{tfree} is implied by the following
\begin{lem}
\[
\text{gcd}(S_n(PE_r),1\leq r\leq n)=\begin{cases} 1,\ &\text{if}\ 2n+1\ \text{is not a prime power}\\
                                                  p,\ &\text{if}\ 2n+1\ \text{is a power of some prime $p$.}\end{cases}
\]
\end{lem}
\begin{proof}
We compute
\begin{equation}\begin{split}\label{r=2}
S_n(PE_r)-S_n(PE_{r+1})&=\left(1+(-1)^{r+1}\binom{2n}{r}\right)-\left(1+(-1)^{r+2}\binom{2n}{r+1}\right)\\
                       &=(-1)^{r+1}\binom{2n+1}{r+1}.
\end{split}\end{equation}
Let $2n+1$ be not a prime power. If $p$ is prime divisor of $2n+1$ then $2n+1=p^sq$ for $q>1$ and $p\nmid q$. We have
\begin{align*}
\binom{p^sq}{p^s}&=\frac{p^sq}{p^s}\cdot\frac{p^sq-1}{p^s-1}\cdot\ldots\cdot\frac{p^sq-p^s+1}{1}\\
                 &\not\equiv 0\mod{p}.
\end{align*}
Hence for $r=p^s-1$ (note $q>1$)
\[
S_n(PE_{p^s-1})-S_n(PE_{p^s})=(-1)^{p^s}\binom{p^sq}{p^s}
\]
is not divisible by $p$. In this way we see that that $\text{gcd}(S_n(PE_r),1\leq r\leq n)\nmid p$ for all prime divisors $p$ of $2n+1$. However, one has $S_n(PE_1)=2n+1$ and thus $\text{gcd}(S_n(PE_r),1\leq r\leq n)=1$.

Now let $2n+1=p^s$ denote a prime power. Then $S_n(PE_1)=2n+1$ is divisible by $p$. From $\binom{p^s}{k}\equiv 0\mod{p}$ for $0<k<p^s$ it follows together with \ref{r=2} that $S_n(PE_r)-S_n(PE_{r+1})\equiv 0\mod{p}$ for $1\leq r\leq n-1$. Hereby we see that $\text{gcd}(S_n(PE_r),1\leq r\leq n)$ is divisible by $p$.

We have
\begin{align*}
\binom{p^s}{p^{s-1}}&=\frac{p^s}{p^{s-1}}\cdot\frac{p^s-1}{p^{s-1}-1}\cdot\ldots\cdot\frac{p^s-p^{s-1}+1}{1}\\
                    &\not\equiv 0\mod{p^2}.
\end{align*}
Hence for $r=p^{s-1}-1$
\[
S_n(PE_{p^{s-1}-1})-S_n(PE_{p^{s-1}})=(-1)^{p^{s-1}}\binom{p^s}{p^{s-1}}
\]
is not divisible by $p^2$. Since $S_n(PE_1)=p^s$ the claim follows.
\end{proof}

\section{The unoriented case}

The first list of generators of the unoriented cobordism ring was given by \cite{thom} (even dimensions) and \cite{dold} (odd dimensions). In \cite{gschnitzer} it is proved that $\RP^2$-bundles serve as generators as well. We shall reprove this result using the developed techniques.
 
For unoriented cobordism we take $G:=SO(3)$, $H:=S(O(2)(1))$ and consider the fiber bundle
\begin{equation}
\RP^2\cong G/H\hto BH\xto{\pi} BG.
\end{equation}
One proceeds like in the oriented case. The corresponding statements to Theorem \ref{a} are
\begin{enumerate}
\item $H^*(B(S(O(2)O(1))))\cong\Z_2[x_1,x_2]$.
\item $H^*(B(SO(3)))\cong \Z_2[y_2,y_3]$.
\item $\pi^*\colon H^*(B(SO(3)))\to H^*(B(S(O(2)O(1))))$ is given by $y_2\mapsto x_1^2+x_2,\ y_3\mapsto x_1x_2$.
\item The total Stiefel-Whitney class of the tangent bundle along the fiber is $\gw(\tau)=1+x_1+x_2$.
\end{enumerate}
The proof is similar to the oriented case.

We note that one need not compute the $Sq^1$-action. However, the integration along the fiber is harder to handle than in the oriented case as one can not compute modulo a simple ideal.

Since $H_*(MO)$ is a free $A_*$-comodule an inspection of the explanation around Theorem \ref{mso} shows that one has to replace $A(0)^*$ by $\Z_2$. We now have to prove the unoriented version of Proposition \ref{l}, namely
\begin{prop}\label{11} The $\Z_2$-comodule map
\[
\overline{T_*}\colon \overline{H_*(MO\wedge\Sigma^2BG_+)} \to \overline{H_*(MO)}
\]
is split surjective onto the augmentation ideal $I(\overline{H_*(MO)})$ of $\overline{H_*(MO)}$.
\end{prop}
Since $\Z_2$ is a field a $\Z_2$-comodule map is split surjective if and only if it is surjective. This in turn means that for the proof of Proposition \ref{11} it remains to show the unoriented version of Proposition \ref{prim}, namely
\begin{prop}\label{12} The map
\[
PH^n(MO) \xto{(Mv)^*} H^n(M(-\tau)) \xto{t^*} H^{n-2}(BG)
\]
is injective.
\end{prop}
\begin{proof} From \cite{switzer}, Theorem 2.80, it follows that $PH^n(MO)$ is isomorphic to $\Z_2$ for $n\neq 2^i-1$ and zero otherwise. We consider the maps (cf. \ref{einfach})
\begin{equation}
\Z_2[\gw_1,\gw_2]\xto{\tilde{v}^*}\Z_2[x_1,x_2]\xto{\pi_!}\Z_2[y_2,y_3].
\end{equation}
The nontrivial primitive element in $PH^n(MO)$ is $s_n(\gw)$ (cf. \ref{primu}). We have $\tilde{v}^*(\gw_1)=x_1$ and $\tilde{v}^*(\gw_2)=x_2$, and $\pi_!$ is a $H^*(BG)$-module map, given on the basis $1,x_1, x_1^2$ by $\pi_!(1)=\pi_!(x_1)=0$ and $\pi_!(x_1^2)=1$.

Let $z_n:=(\pi_!\circ\tilde{v}^*)(s_n(\gw))$ (note that $z_n$ has cohomological degree $n-2$). Modulo 2 one has $s_n(\gw)=\gw_1s_{n-1}(\gw)+\ldots+\gw_{n-1}s_1(\gw)=\gw_1s_{n-1}(\gw)+\gw_2s_{n-2}(\gw)$. Using this formula again for $s_{n-1}(\gw)$ yields
\begin{equation}\label{rf}\begin{split}
 z_n     &=(\pi_!\circ\tilde{v}^*)  (\gw_1(\gw_1s_{n-2}(\gw)+\gw_2s_{n-3}(\gw))+\gw_2s_{n-2}(\gw))\\
         &=(\pi_!\circ\tilde{v}^*)  ((\gw_1^2+\gw_2)s_{n-2}(\gw)+\gw_1\gw_2s_{n-3}(\gw))\\
         &=y_2z_{n-2}+y_3z_{n-3}.
\end{split}\end{equation}

We compute
\begin{equation}\label{ic}
z_1=0,\ z_2=1\ \text{and}\ z_3=0.
\end{equation}
Hence Proposition \ref{12} follows from
\begin{lem}\label{13} The sequence $z_n\in\Z_2[y_2,y_3]$, recursively defined by \ref{rf} with initial condition \ref{ic}, is not zero for $n\neq 2^i-1$.\end{lem}
\begin{proof} We need the formula
\begin{equation}\label{mf}
z_n=y_2^{2^j}z_{n-2^{j+1}}+y_3^{2^j}z_{n-3\cdot2^j},\quad\text{for all}\ j\ \text{with}\ 3\cdot2^j<n,
\end{equation}
which follows easily by induction over $j$: The initial step $j=0$ is given by \ref{rf}. Moreover, assuming the formula for $j-1$, then
\begin{align*}
z_n&=y_2^{2^{j-1}}z_{n-2^j}+y_3^{2^{j-1}}z_{n-3\cdot2^{j-1}}\quad\text{assumption applied to $z_n$}\\
   &=y_2^{2^{j-1}} (y_2^{2^{j-1}}z_{n-2^j-2^j}           + y_3^{2^{j-1}}z_{n-3\cdot2^{j-1}-2^j})\quad\text{assumption applied to $z_{n-2^j}$}\\
   &+y_3^{2^{j-1}} (y_2^{2^{j-1}}z_{n-3\cdot2^{j-1}-2^j} + y_3^{2^{j-1}}z_{n-3\cdot2^{j-1}-3\cdot2^{j-1}})\quad\text{assumption applied to $z_{n-3\cdot 2^{j-1}}$}\\
   &=y_2^{2^j}z_{n-2^{j+1}}+y_3^{2^j}z_{n-3\cdot 2^j}.\\
\end{align*}

Now for any $n\in\N$ pick a natural number $i\geq 0$ such that
\begin{equation}
n\equiv 2^i-1\mod{2^{i+1}}.
\end{equation}
Such an $i$ exists (and it is unique): Consider the $2$-adic expansion of $n$. Say that after the last zero digit there are $m$ 'ones', then $n\equiv 2^m-1\mod{2^{m+1}}$ (for example if $n=19=10011$ then $m=2$ and $19\equiv 3\mod{8}$).
We set $\ga(i,n):={\frac{n-3\cdot (2^i-1)-2}{2}}$ and shall prove by induction over $i$ that for all $n\neq 2^i-1$
\begin{equation}\label{ia}
z_n\equiv y_3^{2^i-1}y_2^{\ga(i,n)}\mod{y_3^{2^i}}.
\end{equation}
Hence, $z_n\neq 0$ as desired.

Initial step: For $i=0$, which means $n$ even, it follows by repeated use of \ref{mf} that
$z_n\equiv y_2^{\frac{n-2}{2}}\mod{y_3}$.

Induction step: Assume that \ref{ia} is proven for $i-1$. In \ref{mf} we set $j=i-1$, the condition $3\cdot2^j<n$ is satisfied since $n\neq 2^i-1$ and $n\equiv 2^i-1\mod{2^{i+1}}$, and get
\begin{equation}\label{le}
z_n=y_2^{2^{i-1}}z_{n-2^i}+y_3^{2^{i-1}}z_{n-3\cdot 2^{i-1}}.
\end{equation}
We show that modulo $y_3^{2^i}$ the left term of \ref{le} is zero and the right term of \ref{le} is $y_3^{2^i-1}y_2^{\ga(i,n)}$.
First note that \begin{equation}\label{sn0}
z_{2^m-1}=0
\end{equation}
for arbitrary $m\in\N$: Setting $j=m-2$ in \ref{mf} yields
\begin{align*}
z_{2^m-1}&=y_2^{2^{m-2}}z_{2^m-1-2^{m-1}}+y_3^{2^{m-2}}z_{2^m-1-3\cdot 2^{m-2}}\\
         &=y_2^{2^{m-2}}z_{2^{m-1}-1}+y_3^{2^{m-2}}z_{2^{m-2}-1}.\\
\end{align*}
By repeated use we can write this expression as some terms times $z_1$ and $z_3$, which both are zero (cf. \ref{ic}).

Left term of \ref{le}: $n-2^i$ can be written as $k\cdot2^{i+1}-1$ for some $k\geq 1$. For $k=1$ we have $z_{2^{i+1}-1}=0$, cf. \ref{sn0}. For $k\geq 2$ we can set $j=i$ in \ref{mf} and get
\begin{align*}
z_{k\cdot2^{i+1}-1}&=y_2^{2^i}z_{k\cdot2^{i+1}-1-2^{i+1}}+y_3^{2^i}\cdot T,\quad T\ \text{some term},\\
                   &\equiv y_2^{2^i}z_{(k-1)\cdot2^{i+1}-1}\mod{y_3^{2^i}}.\\
\end{align*}
By repeated use we conclude that $z_{k\cdot2^{i+1}-1}\equiv y_2^{(k-1)\cdot{2^i}}z_{2^{i+1}-1}\mod{y_3^{2^i}}$ and again with \ref{sn0} the claim follows.

Right term of \ref{le}: Since
\begin{align*}
n-3\cdot 2^{i-1} &\equiv 2^i-1-3\cdot 2^{i-1}\mod{2^{i+1}}\\
                 &\equiv -2^{i-1}-1\mod{2^{i+1}}\\
                 &\equiv 2^{i-1}-1\mod{2^i}\\
\end{align*}
and $n-3\cdot 2^{i-1}\neq 2^{i-1}-1$ we can apply the induction assumption to $z_{n-3\cdot 2^{i-1}}$ and obtain
\begin{align*}
y_3^{2^{i-1}}z_{n-3\cdot 2^{i-1}}&=y_3^{2^{i-1}}(y_3^{2^{i-1}-1}y_2^{\ga(i-1,n-3\cdot 2^{i-1})}+T'\cdot y_3^{2^{i-1}}),\quad T'\ \text{some term},\\
                                 &\equiv y_3^{2^i-1}y_2^{\ga(i,n)}\mod{y_3^{2^i}}.\\
\end{align*}
This finishes the induction step and so the proof of Lemma \ref{13} and in turn the proof of Proposition \ref{12}.
\end{proof}
\end{proof}

It remains to show that the Adams spectral sequence
\[
Ext_{A_*}^{s,t}(\Z_2,H_*(MO\wedge\Sigma^2BG_+))\Longrightarrow\pi_{t-s}(MO\wedge\Sigma^2BG_+)/T_{\nmid 2}
\]
collapses. The unoriented version of \ref{fmo} is given by
\[
H_*(MO)\otimes H_*(\Sigma^2BG_+)\cong A_*\boxempty_{\Z_2}(\overline{H_*(MO)}\ot H_*(\Sigma^2BG_+)).
\]
This means that $H_*(MO)\otimes H_*(\Sigma^2BG_+)$ is free as a $A_*$-comodule and hence the spectral sequence collapses.


\providecommand{\bysame}{\leavevmode\hbox to3em{\hrulefill}\thinspace}
\providecommand{\MR}{\relax\ifhmode\unskip\space\fi MR }
\providecommand{\MRhref}[2]{%
  \href{http://www.ams.org/mathscinet-getitem?mr=#1}{#2}
}
\providecommand{\href}[2]{#2}

\end{document}